\documentclass{article}
\usepackage{graphicx} 
\usepackage[colorlinks]{hyperref}

\usepackage{amsmath}
\usepackage{amsfonts}
\usepackage{latexsym}
\usepackage{amssymb}

\newtheorem{thm}{Theorem}[section]
\newtheorem{theorem}[thm]{Theorem}

\newtheorem{lemma}[thm]{Lemma}

\newtheorem{proposition}[thm]{Proposition}

\newtheorem{definition}[thm]{Definition}

\newtheorem{remark}[thm]{Remark}

\newcommand{\beq}{\begin{equation}}
\newcommand{\eeq}{\end{equation}}
\newcommand{\beqa}{\begin{eqnarray}}
\newcommand{\eeqa}{\end{eqnarray}}
\newcommand{\beqas}{\begin{eqnarray*}}
\newcommand{\eeqas}{\end{eqnarray*}}
\newcommand{\bi}{\begin{itemize}}
\newcommand{\ei}{\end{itemize}}

\newcommand{\argmin}{\mathrm{argmin}\,}

\title{Parallel Polyhedral Projection Method for the Convex Feasibility Problem}

\date{}

\begin{document}

\maketitle

\begin{center}
\begin{tabular}{ccc}

\begin{tabular}{c}
Pablo Barros\\
School of Applied Mathematics, FGV\\
Praia de Botafogo, Rio de Janeiro, Brazil\\
{\tt pabloacbarros@gmail.com}
\end{tabular}\\
\begin{tabular}{c}
Vincent Guigues\\
School of Applied Mathematics, FGV\\
Praia de Botafogo, Rio de Janeiro, Brazil\\
{\tt vincent.guigues@fgv.br}
\end{tabular}
\begin{tabular}{c}
Roger Behling\\
Department of Mathematics, UFSC\\
Blumenau, SC, Brazil\\
{\tt rogerbehling@gmail.com}\\
\end{tabular}
\end{tabular}
\end{center}

\begin{abstract}

In this paper, we introduce and study the Parallel Polyhedral Projection Method (3PM) for finding a point in the intersection of a finite number of closed convex sets. Each iteration consists of two phases: the first involves individual parallel projections, either exact or approximate, onto the target sets; the second performs an exact or approximate projection onto the polyhedron formed by the intersection of supporting half-spaces constructed from these projections. Despite its simplicity, the 3PM strategy appears to be novel, as previous work in the literature has primarily focused on parallel procedures such as Cimmino’s method. Numerically, we observed that when the number of target sets exceeds two, 3PM outperforms several classical and more recent projection-type methods. Theoretically, we establish global convergence of 3PM without regularity assumptions. Under an additional Slater condition or an error bound assumption, we prove linear convergence, even when using approximate projections in both phases. Furthermore, under additional geometric conditions, we show that the method achieves superlinear convergence.

		\\
		
  
\end{abstract}

\section{Introduction}\label{sec1}

The Convex Feasibility Problem involves finding a point within the intersection $U \neq \emptyset$ of a finite collection of closed convex sets \( \{ U_i \}_{i=1}^m \), that is
\begin{equation}
        \mbox{find } u^{\star} \in U := \bigcap_{i=1}^m U_i. \label{CFP}            
\end{equation}

Finding a point in the intersection of finitely many closed convex sets is a fundamental task that serves as a modeling paradigm in numerous computational settings. Projection algorithms have long been prominent for these problems, dating back to Cimmino’s seminal 1938 method of simultaneous (parallel) projections \cite{Cimmino1938}. Cimmino’s algorithm introduced the idea of averaging projections onto all constraint sets at each iteration, an approach that proved influential in early applications such as computerized tomography \cite{Combettes1996}. Another cornerstone is the Douglas–Rachford (DR) method, an iterative reflection scheme that has become a popular solver for convex feasibility and related optimization problems \cite{Bauschke2017, BorweinSims2011, AragónArtacho04072019, Bauschke2016, Phan2016, BauschkeMoursi2016}. 

In recent years, novel enhancements of these classical techniques have emerged \cite{AragonArtacho2018, math12050675, doi:10.1137/21M1453475, doi:10.1137/140961444}. Notably, the circumcentered-reflection method (CRM) was proposed as a geometrically inspired acceleration of projection algorithms, computing the circumcenter of several reflected points to jump closer to the feasibility solution \cite{crm1, BehlingBelloCruzIusemSantos2024, BauschkeOuyangWang2022}. The CRM can dramatically outperform the standard DR method in certain settings (e.g. affine subspaces), enjoying linear convergence under appropriate conditions \cite{crm_block, crm2}. Beyond CRM, many recent works have developed randomized and block-iterative projection schemes that leverage parallel computing architectures and offer improved convergence guarantees \cite{doi:10.1137/18M1167061, doi:10.1137/23M1567503, Kolobov2022}. 

The continued interest in projection methods \cite{scolnik2008incomplete, shen2022penalized, li2016douglas} stems from their wide-ranging applications, including gene regulatory network inference \cite{Wang2017}, large-scale signal processing \cite{combettes2011proximal}, and even combinatorial optimization problems reformulated as feasibility problems \cite{AragonArtacho2014}. Within this rich landscape of algorithms, the \emph{Parallel Polyhedral Projection Method} (3PM) advances the state of the art by adopting a fully simultaneous projection scheme in the spirit of Cimmino’s method, while incorporating modern geometric strategies inspired by the Douglas–Rachford and Circumcentered Reflection methods to enhance robustness and accelerate convergence.

In this work, we introduce the 3PM algorithm, along with its inexact variant A3PM. Each iteration of both methods is structured into two phases:
\begin{enumerate}
  \item Compute the (exact or approximate) projections \( P_{U_1}(x), P_{U_2}(x), \ldots, P_{U_m}(x) \) of the current iterate \( x \) onto each of the sets \( U_i \);
  \item Define the next iterate as the (exact or approximate) projection of \( x \) onto the polyhedron \( \Omega := \bigcap_{i=1}^m S_i \), where each supporting half-space is given by
  \[
  S_i := \left\{ z \in \mathbb{R}^n \mid \left\langle x - P_{U_i}(x), z - P_{U_i}(x) \right\rangle \le 0 \right\}.
  \]
\end{enumerate}
Although conceptually simple, this two-phase strategy is new, to the best of our knowledge.

The main contributions of this work are as follows:\\
\par {\textbf{Parallel Polyhedral Projection Method.}} We propose the \emph{Parallel Polyhedral Projection Method} (3PM), a new approach for solving the convex feasibility problem \eqref{CFP}. At each iteration, 3PM performs parallel projections onto the target sets and then computes a single projection onto a polyhedron defined by affine half-spaces supporting the sets at the projected points. We prove that 3PM converges globally without requiring any regularity assumptions. Moreover, if the intersection \( U \) satisfies an error bound condition, which is guaranteed when \( \operatorname{int}(U) \neq \emptyset \), then the method achieves a linear convergence rate. Under additional geometric conditions, we show that 3PM exhibits superlinear convergence.\\

\par {\textbf{Inexact Variant A3PM.}} We introduce A3PM, an inexact version of 3PM in which both the parallel and polyhedral projections may be computed approximately. This generalization retains the convergence guarantees of 3PM: global convergence is preserved even without regularity, and linear convergence is ensured under standard error bound conditions. This makes A3PM suitable for applications where exact projections are impractical or costly. \\

\par {\textbf{Numerical experiments.}} We present numerical experiments on a variety of convex feasibility problems. The results highlight the competitive performance of 3PM and A3PM in comparison to classical and modern projection methods such as CRM, P-CRM, and F-SPM, particularly in settings with more than two sets. \\

This paper is organized as follows. Section \ref{sec2} is devoted entirely to the Parallel Polyhedral Projection Method (3PM): we describe its algorithmic structure, analyze its convergence properties, and relate it to classical projection methods. Section \ref{sec3} introduces A3PM, an inexact variant of 3PM that allows for approximate projections in both algorithmic phases. We establish convergence guarantees for A3PM under mild assumptions, and demonstrate that its theoretical properties mirror those of the exact version. Finally, in Section \ref{sec4}, we present numerical experiments that benchmark the performance of 3PM and A3PM against a selection of classical and recent projection algorithms.

\subsection{Basic definitions and notation}

We denote by \( \mathbb{R}^n \) the standard \( n \)-dimensional Euclidean space, equipped with the inner product \( \langle \cdot, \cdot \rangle \) and associated norm \( \| \cdot \| \). 
Given a nonempty, closed convex set \( C \subset \mathbb{R}^n \), the projection of a point \( x \in \mathbb{R}^n \) onto \( C \) is denoted by \( P_C(x) \), and is defined as the unique point in \( C \) closest to \( x \), that is, 
\[P_C(x) := \underset{y \in C}{\argmin} \|x - y\|.\] 
The distance from \( x \) to \( C \) is denoted by 
\[ \operatorname{dist}(x, C) := \inf_{y \in C} \|x - y\|.\]

Throughout the paper, we assume that the sets \( U_1, \dots, U_m \subset \mathbb{R}^n \) are closed, convex, and nonempty. Their intersection is denoted by \( U := \bigcap_{i=1}^m U_i \), which we assume is  nonempty.

\section{Parallel Polyhedral Projection Method for convex feasibility problems (3PM)}\label{sec2}

\subsection{3PM method}

In this section, we describe a new cutting plane method
for the convex feasibility problem \eqref{CFP}.
We start from an arbitrary point $x_1=x$ and for iteration
$k \geq 1$, given $x_{k}$, we compute 
the projection
$P_{U_i}(x_k)$ of $x_k$ onto $U_i$ for every $i$.
Then let
\begin{equation}\label{sik}
S_{ik} = \Big\{z : \left \langle x_k - P_{U_i}(x_k),z-P_{U_i}(x_k) \right \rangle
\leq 0 \Big\}
\end{equation}  
be a half space
that contains $U_i$, built from a separating hyperplane
for $U_i$, and let 
\begin{equation}\label{omk}
\Omega_k=\bigcap_{i=1}^m S_{ik}.
\end{equation}
We compute the next iterate $x_{k+1}$
as 
$x_{k+1}=P_{\Omega_k}(x_k).
$

The steps of 3PM are summarized below.

\noindent\rule[0.5ex]{1\columnwidth}{1pt}
	
	3PM
	
	\noindent\rule[0.5ex]{1\columnwidth}{1pt}
	\begin{itemize}
		\item [0.] Let  $k=1$, $x_1=x \in \mathbb{R}^n$, and $U_1, U_2, \ldots, U_m$, be closed convex sets. If $x_1 \in U$ stop. Otherwise go to step 1.
	    \item[1.] Compute $P_{U_1}(x_k)$, $P_{U_2}(x_k)$, $\ldots$, $P_{U_m}(x_k)$.
		\item[2.] Define $S_{ik}$ by \eqref{sik} and
        $\Omega_k$ by
        \eqref{omk}. Compute 
  $$
  P_{\Omega_k}(x_k).
  $$
		\item[3.] Do $x_{k+1} \leftarrow P_{\Omega_k}(x_k)$, $k \leftarrow k+1$. If $x_{k+1} \in U$ stop, otherwise
go to step 1.
	\end{itemize}
	\rule[0.5ex]{1\columnwidth}{1pt}

 3PM computes the sequence of iterates $(x_k)$ given by   
    \begin{equation}\label{defpomk}
x_{k+1}=P_{\Omega_k}(x_k)
\end{equation}
for $\Omega_k$ given by \eqref{omk}.

This two-step process decouples the iteration into a first phase of identifying plausible corrections via individual projections, followed by a second phase in which the algorithm consolidates this information by projecting onto the polyhedron \(\Omega_k\). This region can be seen as a local outer approximation of the feasible set, built from first-order information about each \(U_i\). The final projection thus performs a coordinated movement informed by all constraints, offering a sharper descent than methods based on single-set updates.

Figure \ref{fig:3PM} illustrates the computation of an iteration of 3PM for 3 sets $U_1$, $U_2$, and $U_3$
in $\mathbb{R}^2$. In this figure we represent:
\begin{itemize}
\item convex sets $U_1$, $U_2$, and $U_3$ and their intersection $U$;
\item an arbitrary initial point $x$;
\item the projections $P_{U_1}(x)$, $P_{U_2}(x)$, and
$P_{U_3}(x)$ of $x$ onto respectively $U_1$, $U_2$, and
$U_3$;
\item half-spaces $S_1$, $S_2$, and $S_3$ which play the role of half-spaces $S_{ik}$ in 3PM; we dropped the iteration index
for simplicity;
\item the intersection $\Omega=S_1 \cap S_2 \cap S_3$ which plays the role of $\Omega_k$ in 3PM;
\item the next iterate which is the projection 
$P_{\Omega}(x)$ of $x$ onto $\Omega$ which is
closer (as shown in Lemma \ref{lemmfejer} below)
than $x$ (much closer in this example) to any point
in the intersection $U:=U_1\cap U_2\cap U_3$.
\end{itemize}

\begin{figure}
    \centering
\includegraphics[width=0.95\textwidth]{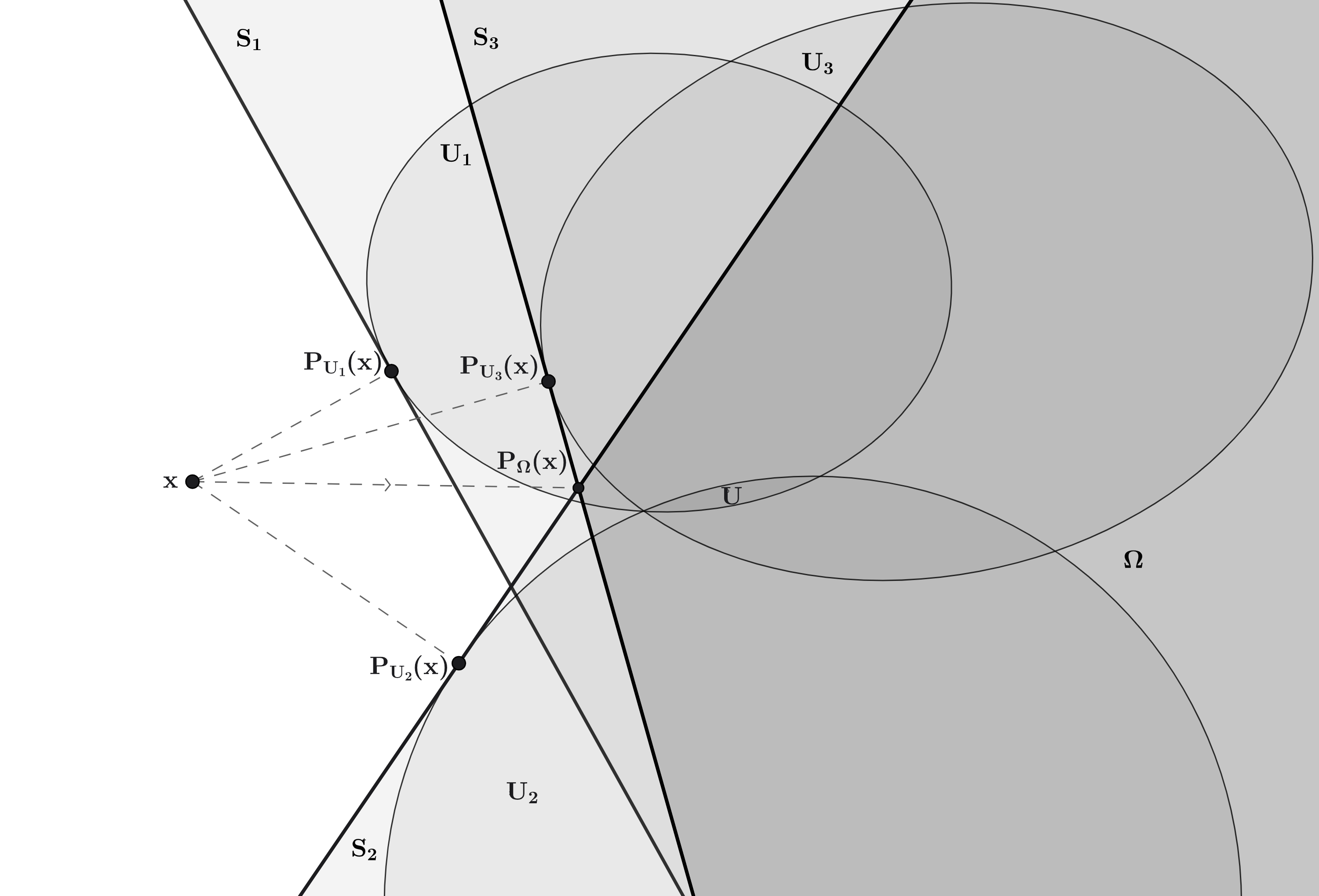}
    \caption{An iteration of 3PM}
    \label{fig:3PM}
\end{figure}

\subsection{Convergence analysis}

It should be noted that if $x_k$ already belongs to $U_i$ then 
$S_{i k}=\mathbb{R}^n$.

\subsubsection{Convergence proof}

We begin by recalling a fundamental property of projections onto closed convex sets, which characterizes the projection as the unique point satisfying a first-order optimality condition. This inequality plays a central role in demonstrating geometric properties of projection-based algorithms such as 3PM.

\begin{proposition}\label{propproj} \emph{(Projection characterization)\textbf{.}} Let $X$ be  a closed nonempty convex set. Then the projection
$P_X(z)$ of $z \in \mathbb{R}^n$ onto $X$ satisfies
$$
\langle z-P_{X}(z),x-P_{X}(z)\rangle \leq 0, \quad \forall \, x \in X.
$$
\end{proposition}
\begin{proof}
    See, e.g., Theorem 3.16 in \cite{Bauschke2017}.
\end{proof}

We now turn to the behavior of the sequence \( (x_k) \) generated by 3PM. The following lemma shows that this sequence is Fejér monotone with respect to the feasible region \( U \), meaning that the distance from each iterate to any point in \( U \) does not increase. This property provides a key tool for establishing convergence results.

\begin{lemma}\label{lemmfejer} \emph{(Fejér monotonicity of 3PM)\textbf{.}}
For any starting point $x_1 \in \mathbb{R}^n$, the sequence $(x_k)$ generated by 3PM satisfies 
\begin{equation}\label{fej1}
\|x_{k}-u\|^2 \geq
\|x_{k}-x_{k+1}\|^2 + \|x_{k+1}-u\|^2
\end{equation}
and therefore 
is Fejér monotone with respect to 
$U$, i.e.,
\begin{equation}\label{xkp1xk}
\|x_{k+1}-u\| \leq \|x_{k}-u\|,\;\forall \, u \in U.
\end{equation}
\end{lemma}
\begin{proof}
By Proposition \ref{propproj}, we have 
\begin{equation}\label{projxkp1}
\langle x_k-x_{k+1},u-x_{k+1}\rangle \leq 0,\;\;\forall \, u \in \Omega_k.
\end{equation}
We also have 
\begin{equation}\label{propst}
u  \in U \Rightarrow u \in U_i, \forall \, i \Rightarrow  
u \in \Omega_k,
\end{equation}
which combined with \eqref{projxkp1}
gives for all $u \in U$,
\begin{equation}
\begin{array}{lcl}
\|x_{k}-u\|^2 & = & \|x_{k}-x_{k+1}\|^2 + \|x_{k+1}-u\|^2 +
2 \langle x_{k}-x_{k+1}, x_{k+1}-u \rangle \\
& \stackrel{\eqref{projxkp1}}{\geq }  &
\|x_{k}-x_{k+1}\|^2 + \|x_{k+1}-u\|^2
\end{array}
\end{equation}
and therefore \eqref{fej1}, \eqref{xkp1xk} follow.
\end{proof}

Fejér monotonicity guarantees that the iterates of 3PM never move further from the feasible region. This property, shared by many classical projection schemes, is particularly useful here as it ensures that even approximate polyhedral steps do not destabilize the sequence.

We now derive a consequence of this property.

\begin{proposition}\label{fejerconv} \emph{(Distance bound from Fejér monotonicity)\textbf{.}}
    Let $(z_k)$ be a sequence converging to a point $\bar{z}.$ Suppose that $(z_k)$ is Fejér monotone with respect to a closed, convex set $X.$ Then, for any $k$, it holds that 
    \begin{equation}\label{limdistbound}
        \displaystyle \emph{dist}(z_k,X) \geq \frac{1}{2}\|z_k - \bar{z}\|.
    \end{equation}
\end{proposition}

\begin{proof}
    By Fejér monotonicity, for any $x \in X$ and $k$,
\begin{equation*}
\|z_{\ell}-x\| \leq \|z_{k}-x\|\quad \forall \, \ell \ge k.
\end{equation*}
Taking the limit as \(\ell \to +\infty\), we get
\begin{equation*}
\|\bar{z}-x\| \leq \|z_{k}-x\|.
\end{equation*}
Substituting $x = P_{X}(z_k)$ above, we conclude that
\begin{equation*}
\|\bar{z}-P_X(z_k)\| \leq \text{dist}(z_k,X).
\end{equation*}
Thus
\begin{equation*}
    \begin{array}{lcl}
    \|z_k - \bar{z}\| &\leq& \|z_k - P_Z(z_k)\| + \|P_Z(z_k) - \bar{z}\|
    \\ &\leq& 2 \, \text{dist}(z_k,X),
    \end{array}
\end{equation*}
yielding \eqref{limdistbound}.
\end{proof}

The next lemma shows that the distance from each iterate to every individual set \( U_i \) vanishes as \( k \to \infty \).

\begin{lemma}\label{lemmadist} \emph{(Vanishing distance to individual sets)\textbf{.}} For any starting point $x_1 \in \mathbb{R}^n$, the sequence $(x_k)$ generated by 3PM satisfies
\begin{equation}\label{first}
\displaystyle \lim_{k \rightarrow +\infty} 
 \max_{i=1,\ldots,m} \emph{dist}(x_k,U_i) =0.
\end{equation}
\end{lemma}
\begin{proof} The inclusion
$\Omega_k \subset S_{ik}$ for all $i$ implies 
$$
\|x_k-x_{k+1}\|=\text{dist}(x_k, \Omega_k) \ge \text{dist}(x_k, S_{ik}) = \|x_k - P_{U_i}(x_k)\| = \text{dist}(x_k,U_i).
$$
for all $i$ and therefore 
$$
\|x_k-x_{k+1}\| \ge \displaystyle \max_{i=1,\ldots,m} \text{dist}(x_k,U_i),
$$
which implies 
\begin{equation}\label{conveq}
    \begin{array}{lcl}
    \|x_{k}-u\|^2 & \stackrel{\eqref{fej1}}{\geq} &
    \|x_{k}-x_{k+1}\|^2 + \|x_{k+1}-u\|^2 \\   
    & \geq & \displaystyle  \max_{i=1,\ldots,m} \left \{ \text{dist}(x_k,U_i)^2 \right\} + \|x_{k+1}-u\|^2.
\end{array}
\end{equation}
Therefore, the series $S_n=\sum_{k=1}^n s_k$ with general term
$s_k=\displaystyle \max_{i=1,\ldots,m} \left \{ \text{dist}(x_k,U_i)^2 \right\}$
is nondecreasing, bounded by $\|x_1-u\|^2$ (for any $u \in U$)
and therefore converges, implying that $s_k$ converges
to 0.
\end{proof}

Having shown that the iterates become increasingly feasible, we now address the limiting behavior of the 3PM sequence \((x_k)\).

\begin{lemma}\label{bdedappproj} \emph{(Feasibility of accumulation points)\textbf{.}} For any starting point $x_1 \in \mathbb{R}^n$, the sequence $(x_k)$ generated by 3PM is bounded and any accumulation
point of this sequence belongs to  $U$.
\end{lemma}
\begin{proof}
Since $(x_k)$ is Fejér monotone with respect to $U$,
we have that $x_k$ is bounded:
$$
\|x_k\| \leq \|u\| + \|x_1-u\|
$$
for all $u \in U$. Therefore, it has a convergent subsequence.
Let $\bar x$ be an accumulation point of $(x_k)$ and let
$\varphi$ such that $x_{\varphi(k)}$ converges to $\bar x$.
Then by Lemma \ref{lemmadist},  
$$
\displaystyle \lim_{k \rightarrow +\infty} 
 \max_{i=1,\ldots,m} \mbox{dist}(x_{\varphi(k)},U_i) =0
$$
and by continuity of the distance function, we obtain
$$
\max_{i=1,\ldots,m} \mbox{dist}(\bar x,U_i) =0
$$
meaning that $\bar x \in U$.
\end{proof}

The following theorem completes the global convergence analysis of 3PM. Building on the Fejér monotonicity, asymptotic feasibility, and boundedness of the sequence, we now establish that the iterates \( (x_k) \) converge to a point in the intersection \( U \).

\begin{theorem}\label{conv3PM} \emph{(Global convergence of 3PM)\textbf{.}} For any starting point $x_1 \in \mathbb{R}^n$, the sequence $(x_k)$ generated by 3PM converges to a point 
$\bar x \in U$.
\end{theorem}
\begin{proof} Since $(x_k)$ is bounded it has accumulation points.
We show that there is only one accumulation point.
Consider two accumulation points $\bar x$ and $\bar y$, which by the previous lemma belong to $U$,
and take subsequences $x_{\varphi(k)}$ and $x_{\psi(k)}$
such that $\bar x=\displaystyle \lim_{k \rightarrow +\infty} x_{\varphi(k)}$ and $\bar y=\displaystyle \lim_{k \rightarrow +\infty} x_{\psi(k)}$. Then fixing $\ell$, for every $k$ sufficiently large, we have by \eqref{xkp1xk} that
$$
\|\bar x - x_{\psi(k)}\| \leq \|\bar x-x_{\varphi(\ell)}\|.
$$
Taking the limit in the previous inequality, first with
respect to $\ell \rightarrow +\infty$ and then with respect
to $k \rightarrow +\infty$ we obtain
$\|\bar x - \bar y\| \leq 0$ and therefore $\bar x=\bar y$, which achieves the proof.
\end{proof}

\subsubsection{Linear convergence of 3PM}

This section demonstrates the linear convergence of the 3PM algorithm. To that end, we introduce a standard regularity assumption on the collection of constraint sets, known as the error bound (EB) condition. This condition plays a central role in the convergence theory of projection algorithms, linking the individual distances to the sets \( U_i \) with the distance to their intersection.

The error bound condition is also referred to in the literature as linear regularity \cite{BauschkeBorwein1996, BauschkeBorwein1993} or subtransversality \cite{Kruger2018}, and it has been widely employed in the analysis of first-order methods for convex feasibility and related problems.

\begin{definition}\label{defEB} \emph{(Error bound condition)\textbf{.}}
Let \( \{U_i\}_{i=1}^m \) be closed convex sets in \( \mathbb{R}^n \) with nonempty intersection \( U := \bigcap_{i=1}^m U_i \neq \emptyset \). We say that the family \( \{U_i\} \) satisfies a local error bound condition at \( \bar{x} \in U \) if there exist a constant \( \omega \in (0,1) \) and a neighborhood \( V \) of \( \bar{x} \) such that
\begin{equation} \label{EB}\tag{EB}
\omega\, \emph{dist}(x, U) \leq \max_{1 \leq i \leq m} \emph{dist}(x, U_i), \quad \forall \, x \in V.
\end{equation}
\end{definition}

The error bound condition ensures that no point in a neighborhood of \( \bar{x} \) can be simultaneously close to all the sets \( U_i \) while being significantly far from their intersection. The constant \( \omega \) captures the degree of regularity in the configuration of the sets: when the sets meet transversally, this constant is uniformly bounded away from zero. In a geometric sense, \( \omega \) can be thought of as reflecting the “angle” between the sets at the point of intersection, with sharper angles corresponding to smaller values of \( \omega \).

This assumption will serve as a foundation for our convergence rate analysis. In particular, we will show that the EB condition implies linear convergence of the iterates generated by 3PM, and we derive an explicit upper bound on the corresponding asymptotic rate. Furthermore, under additional smoothness and structural assumptions, this analysis is extended to establish superlinear convergence.

We are now in a position to show that the error bound condition implies a linear convergence rate for the 3PM algorithm. The result below shows that, under assumption \ref{EB}, the sequence \( (x_k) \) generated by 3PM converges to a solution at a linear rate, and provides an explicit bound on the asymptotic contraction factor.

\begin{theorem}\label{linconv3PM} \emph{(Linear convergence of 3PM)\textbf{.}} Let $(x_k)$ be the sequence generated by 3PM for a starting point $x_1 \in \mathbb{R}^n$. Assume that \ref{EB} holds for $ \{U_i\}_{i=1}^m $ around the limit point $\bar x$ of $(x_k)$. Then $(x_k)$ converges linearly to $\bar x$.
\end{theorem}
\begin{proof} By \ref{EB}, for some $\omega \in (0,1)$, we have
\begin{equation}
    \displaystyle \omega \, \text{dist}(x_k,U) \leq \max_{i=1,\ldots,m} \text{dist}(x_k,U_i)
\end{equation}
for every sufficiently large $k$. By Proposition \ref{fejerconv}, we have
\begin{equation}\label{rlindist}
    \displaystyle \text{dist}(x_k,U) \geq \frac{1}{2}\|x_k - \bar{x}\|
\end{equation}
for any $k$. Now from \eqref{conveq} with $u = \bar{x}$, 
\begin{equation}\label{linconveq1}
    \begin{array}{lcl} 
    \|x_{k}-\bar{x}\|^2 & \geq & \displaystyle \max_{i=1,\ldots,m} \left \{ \text{dist}(x_k,U_i)^2 \right\} + \|x_{k+1}-\bar{x}\|^2
    \\ & \geq &  \omega^2 \, \text{dist}(x_k,U)^2 + \|x_{k+1}-\bar{x}\|^2
    \\ & \geq & \displaystyle \frac{\omega^2}{4}\|x_k - \bar{x}\|^2 + \|x_{k+1}-\bar{x}\|^2.
\end{array}
\end{equation}
Rearranging the terms, we arrive at
\begin{equation}\label{linconveq2}
\displaystyle \sqrt{1-\frac{\omega^2}{4}} \,\|x_{k}-\bar{x}\| \geq \|x_{k+1}-\bar{x}\|,
\end{equation}
proving the linear convergence of 3PM.
\end{proof}

The result shows that 3PM not only converges but does so with a guaranteed rate when the constraint sets satisfy a mild regularity condition. In particular, the rate constant depends explicitly on the error bound constant \( \omega \). This analysis also lays the groundwork for understanding regimes where the convergence may be faster than linear, as we explore next.

\subsubsection{Superlinear convergence of 3PM}

In this section, we establish the superlinear convergence of the 3PM sequence \( (x_k) \) under the assumptions stated in Theorem \ref{thsplc}. In the context of the 3PM algorithm, we denote the projection of \( x_k \) onto \( U_i \) by \( p_{ik} := P_{U_i}(x_k) \), and define the residual vector \( d_{ik} := x_k - p_{ik} \).

The assumptions for superlinear convergence include the nonemptiness of the interior of the intersection \( U := \bigcap_{i=1}^m U_i \), and the requirement that each boundary \( \partial U_i \) is locally a differentiable manifold. These conditions imply that each \( \partial U_i \) is a \( (n - 1) \)-dimensional manifold.

When the boundary of each set \( U_i \) is locally smooth, the supporting half-spaces constructed in 3PM can be interpreted as local first-order linearizations of the sets. In this regime, the polyhedron \( \Omega_k \) formed by these half-spaces approximates the intersection \( U = \bigcap_i U_i \) increasingly well as \( x_k \) converges to a point in \(U\). 

In particular, when the iterates \( x_k \) eventually avoid lying too close to the boundary of certain “inactive” sets (e.g., those not containing \( \bar{x} \) on their boundary), the supporting half-spaces corresponding to the remaining sets tend to align in a configuration that more accurately captures the local geometry of the feasible region. This selective stabilization has the effect of making the correction step in 3PM increasingly well-directed, exhibiting behavior reminiscent of Newton-like superlinear acceleration.

We begin by presenting Lemmas \ref{lemdistfunc} and \ref{lemmanif}, which will be instrumental in the proof of superlinear convergence of 3PM.

\begin{lemma}\label{lemdistfunc} \emph{(Monotonicity of projection-distance ratio)\textbf{.}} Let \( X \subset \mathbb{R}^n \) be a nonempty closed convex set, let \(x \not \in X\) and \( p:=P_X(x) \). Let any \( q \in X \). Then the function
\[
\phi(\mu) := \frac{\emph{dist}(p + \mu(x - p), \, X)}{\|p + \mu(x - p) - q\|}
\]
is increasing in \( \mu > 0 \).
\end{lemma}

\begin{proof} Since $P_X(p + \mu(x - p)) = p,$ we have
\[
\phi(\mu) = \frac{\mu \|x - p\|}{\|p - q + \mu(x - p)\|}.
\]
Define \( a := x - p \) and \( b := p - q \). Then
\[
\phi(\mu)^2 = \frac{\mu^2 \|a\|^2}{\|b + \mu a\|^2} = \frac{\mu^2 \|a\|^2}{\|b\|^2 + 2\mu \langle a, b \rangle + \mu^2 \|a\|^2} = \frac{\|a\|^2 }{\|b\|^2 \frac{1}{\mu^2} + 2 \langle a, b \rangle \frac{1}{\mu} + \|a\|^2}.
\]
Since \(\langle a, b \rangle \ge 0\) by Proposition \ref{propproj}, the denominator is increasing in $\frac{1}{\mu},$ then decreasing in $\mu$, hence $\phi$ is increasing in $\mu.$
\end{proof}

As mentioned earlier, the nonemptiness of the interior of \( U \), along with the assumption that the boundaries of \( U_i \) are locally differentiable manifolds, implies that these manifolds have dimension \( n-1 \).

Let \( M \subset \mathbb{R}^n \) be any differentiable manifold of dimension \( n-1 \) and \( \bar{z} \in M \). It is well-known that \( M \) can be locally written as
\[
M = \{ z \in \mathbb{R}^n \mid g(z) = 0 \}
\]
for some function \( g: \mathbb{R}^n \to \mathbb{R} \) of class \( C^1 \) with \( g(\bar{z}) = 0 \) and \( \nabla g(\bar{z}) \neq 0 \), so that the tangent hyperplane is given by
\[
T_M(\bar{z}) := \{ z \in \mathbb{R}^n \mid \langle \nabla g(\bar{z}), z - \bar{z} \rangle = 0 \}.
\]

The following lemma is a vital geometric tool in our analysis. It captures the behavior of sequences approaching a smooth manifold and will play a key role in showing the superlinear convergence of 3PM.

\begin{lemma} \label{lemmanif} \emph{(Tangency limit to manifold)\textbf{.}}
Let \( M \subset \mathbb{R}^n \) be a differentiable manifold of dimension \( n-1 \) and \( \bar{z} \in M \). Suppose that both the sequences $(q_k) \subset M$ and $(z_k) \subset \mathbb{R}^n$ converge to $\bar z$ with $z_k \in T_M(q_k)$ and $z_k \neq q_k$. Then
\[
\lim_{k \to \infty} \frac{\operatorname{dist}(z_k, M)}{\| z_k - q_k \|} = 0.
\]
\end{lemma}

\begin{proof} Take $g$ as described above. Evidently, we only have to consider the case $g(z_k) \neq 0$. For all $k$ large enough, by the fundamental theorem of calculus we have
\[
g(z_k) - g(q_k) = g(z_k) = \int_0^1 \left\langle z_k - q_k, \nabla g(q_k + t(z_k - q_k)) \right\rangle dt.
\]
By assumption, \( \langle z_k - q_k, \nabla g(q_k) \rangle = 0 \). Hence
\[
g(z_k) = \int_0^1 \left\langle z_k - q_k,\; \nabla g(q_k + t(z_k - q_k)) - \nabla g(q_k) \right\rangle dt.
\]
Thus
\[
|g(z_k)| \le \int_0^1 \left\| z_k - q_k \right\| \cdot \left\| \nabla g(q_k + t(z_k - q_k)) - \nabla g(q_k) \right\| dt.
\]
Since \(\nabla g\) is continuous, we have
\[
\left\| \nabla g(q_k + t(z_k - q_k)) - \nabla g(q_k) \right\| \to 0
\]
uniformly in \( t \in [0,1] \), which implies
\begin{equation} \label{manifeq1}
\lim_{k \to \infty} \frac{|g(z_k)|}{\|z_k - q_k\|} = 0.
\end{equation}
Define
\[
\tilde{z}_k := z_k - g(z_k) \frac{2}{\|\nabla g(\bar{z})\|^2} \nabla g(z_k).
\]
Then clearly \( \tilde{z}_k \to \bar{z} \). Using the fundamental theorem of calculus again, we get
\begin{align*}
    g(z_k) - g(\tilde{z}_k) &= \int_0^1 \left\langle \nabla g\left(\tilde z_k + t (z_k - \tilde z_k) \right), \; g(z_k) \frac{2}{\|\nabla g(\bar{z})\|^2} \nabla g(z_k) \right\rangle dt. \\
&= g(z_k) \frac{2}{\|\nabla g(\bar{z})\|^2} \int_0^1 \left\langle \nabla g\left(z_k + t (z_k - \tilde z_k)\right), \nabla g(z_k) \right\rangle dt.
\end{align*}
The last integrand converges to \( \|\nabla g(\bar{z})\|^2 > 0\), so for all large \( k \) we have
\[
\int_0^1 \left\langle \nabla g\left(z_k + t (z_k - \tilde z_k)\right), \nabla g(z_k) \right\rangle dt \ge \frac{\|\nabla g(\bar{z})\|^2}{2},
\]
which implies
\[
\frac{g(z_k) - g(\tilde{z}_k)}{g(z_k)} \ge 1,
\]
meaning $g(\tilde{z}_k)$ and $g(z_k)$ have different signs. By continuity of \( g \), there exists \( \lambda_k \in [0,1] \) such that
\[
g\left( z_k - \lambda_k g(z_k) \frac{2}{\|\nabla g(\bar{z})\|^2} \nabla g(z_k) \right) = 0, \quad \text{i.e.}, \quad
z_k - \lambda_k g(z_k) \frac{2}{\|\nabla g(\bar{z})\|^2} \nabla g(z_k) \in M.
\]
Thus
\[
\operatorname{dist}(z_k, M) \le \left\|\lambda_k g(z_k) \frac{2}{\|\nabla g(\bar{z})\|^2} \nabla g(z_k) \right\| \le |g(z_k)| \frac{2\|\nabla g(z_k)\|}{\|\nabla g(\bar{z})\|^2}.
\]
Since \( \|\nabla g(z_k)\| \to \|\nabla g(\bar{z})\| > 0 \), we conclude that
\[
\operatorname{dist}(z_k, M) = \mathcal{O}(|g(z_k)|).
\]
Combining this fact with \eqref{manifeq1}, we obtain:
\begin{equation} \label{manifeq2}
\lim_{k \to \infty} \frac{\operatorname{dist}(z_k, M)}{\|z_k - q_k\|} = 0.
\end{equation}
This completes the argument.
\end{proof}

The following lemma captures the asymptotic alignment of projection directions near smooth boundaries.

\begin{lemma}\emph{(Alignment of projection directions)\textbf{.}} \label{lemCosine}
Let $U \subset \mathbb{R}^n$ be closed, convex, with nonempty interior, and suppose $\partial U$ is locally a differentiable manifold. Let $\bar{x} \in \partial U$, and let $x_k \notin U$ with $x_k \to \bar{x}$. Define $p_k := P_U(x_k)$. Then
$$
\lim_{k \to \infty} \cos \angle(x_k-p_k, x_{k+1}-p_{k+1}) = 1.
$$

\end{lemma}
\begin{proof}
Since $x_k \notin U$, we have $p_k \in \partial U$ and $x_k - p_k \in N_U(p_k)$, so:
$$
\frac{x_k - p_k}{\|x_k - p_k\|} \in N_U(p_k) \cap \mathbb{S}^{n-1}.
$$
By local differentiability, there exists an open ball $B$ 
containing
$\bar{x}$ and a function $g \in C^1(B)$ such that:
$$
U \cap B = \{ x \in B : g(x) \le 0 \}, \quad \partial U \cap B = \{ x \in B : g(x) = 0 \}, \quad \nabla g(x) \ne 0 \text{ on } B.
$$
For all sufficiently large $k$, $p_k \in B$, so:
$$
N_U(p_k) = \mathbb{R}_+ \nabla g(p_k) \quad \text{and} \quad \frac{x_k - p_k}{\|x_k - p_k\|} = \frac{\nabla g(p_k)}{\|\nabla g(p_k)\|}.
$$
Since $\nabla g$ is continuous and $p_k \to \bar{x}$, we conclude:
$$
\frac{\nabla g(p_k)}{\|\nabla g(p_k)\|} \to \frac{\nabla g(\bar{x})}{\|\nabla g(\bar{x})\|} = n(\bar{x}),
$$
where $n(\bar{x})$ is the outward unit normal to $\partial U$ at $\bar{x}$. Thus it becomes evident that
$$\cos \angle(x_k-p_k, x_{k+1}-p_{k+1}) = \left \langle \frac{x_k - p_k}{\|x_k - p_k\|}, \frac{x_{k+1} - p_{k+1}}{\|x_{k+1} - p_{k+1}\|} \right \rangle = \left \langle \frac{\nabla g(p_k)}{\|\nabla g(p_k)\|}, \frac{\nabla g(p_{k+1})}{\|\nabla g(p_{k+1})\|} \right \rangle \to 1.$$
\end{proof}

Finally, we are ready to formalize a result that has been discussed throughout the paper: under a geometric condition, the convergence of 3PM improves beyond linear. Specifically, if the iterates eventually avoid the boundary sets active at the limit point, the polyhedral correction becomes increasingly accurate, yielding superlinear convergence.

\begin{theorem}\label{thsplc} \emph{(Superlinear convergence of 3PM)\textbf{.}} Assume that \(\emph{int} (U) \neq \emptyset\) and each boundary \( \partial U_i \) is locally a differentiable manifold. Let $(x_k)$ be the sequence generated by 3PM for a starting point $x_1 \in \mathbb{R}^n$, which converges to $\bar x$. Define $I = \{i \mid \bar{x} \not \in \emph{int} (U_i)\}$. Assume that $x_k \not \in \bigcup_{i \in I} U_i$ for all $k$ large enough. Then $x_k$ converges to $\bar x$ superlinearly.
\end{theorem}

\begin{proof} 
We have $\displaystyle x_k \in \bigcap_{i \not \in I} U_i$ for all $k$ large enough. From this point on, the sequence is defined only by $\{U_i \mid i \in I\}$. Moreover, it is clear that $\bar x \in \bigcap_{i \in I} \partial U_i.$
For each $i \in I$, Lemma \ref{lemCosine} implies
\begin{equation}\label{costo1}
\cos \angle (d_{i k+1},\, d_{i k}) \to 1.
\end{equation}
Choose $i$ satisfying $\displaystyle \|d_{i k+1}\| = \max_{j \in I} \|d_{j, k+1}\|.$ Let
\[
y_k := x_{k+1} + \mu_k d_{i k+1},
\]
where
\[
\mu_k := \frac{\left\langle p_{i k} - x_{k+1},\, d_{i k} \right\rangle}{\left\langle d_{i  k+1},\, d_{i k} \right\rangle}.
\]
Then for all $k$ large enough, \(\mu_k \ge 0\) and
    \[
    \left \langle y_k - p_{i k},\, d_{i k} \right \rangle = \left \langle x_{k+1} - p_{i k} + \mu_k d_{i k+1},\, d_{i k} \right \rangle = -\left \langle p_{i k} - x_{k+1}, d_{i k} \right \rangle + \mu_k \left \langle d_{i k+1}, d_{i k} \right \rangle = 0.
    \]
That means \(y_k \in \partial S_{ik}\).
Notice that 
\begin{align*}
\|y_k - x_{k+1}\| 
&= \mu_k \|d_{i k+1}\| \\
&= \frac{\left\langle p_{i k} - x_{k+1},\, d_{i k} \right\rangle}{\left\langle d_{i k+1},\, d_{i k} \right\rangle} \cdot \|d_{i k+1}\| \\
&= \frac{\|p_{i k} - x_{k+1}\| \cdot \|d_{i k}\| \cdot \cos \angle (p_{i k} - x_{k+1},\, d_{i k})}{\|d_{i k+1}\| \cdot \|d_{i k}\| \cdot \cos \angle (d_{i k+1},\, d_{i k})} \cdot \|d_{i k+1}\| \\
&= \|p_{i k} - x_{k+1}\| \cdot \frac{\cos \angle (p_{i k} - x_{k+1},\, d_{i k})}{\cos \angle (d_{i k+1},\, d_{i k})}.
\end{align*}
Result \eqref{costo1}, together with the convergences \(\|p_{ik}-x_{k+1}\|\to 0\) and \(x_{k+1} \to \bar{x}\), yields \(y_k\to\bar x\). Clearly we also have \(p_{i k} \to \bar{x}\). From the tangency of \(\partial S_{ik}\) to \(U_i\) at \(p_{i k}\) \cite[Theorems 23.2 and 25.1]{rockafellar1997convex} and the fact that \(y_k \in \partial S_{ik}\), we can apply Lemma \ref{lemmanif} to conclude that
\begin{equation}\label{210limit}
    \lim_{k \to \infty} \frac{\operatorname{dist}(y_k, U_i)}{\| y_k - p_{i k}\|} = 0.
\end{equation}
Since $\mu_k \ge 0$, by Lemma \ref{lemdistfunc}, we get
\begin{equation} \label{210lemma28}
\frac{\operatorname{dist}(x_{k+1}, U_i)}{\| x_{k+1} - p_{i k}\|} \le \frac{\operatorname{dist}(y_k, U_i)}{\| y_k - p_{i k}\|}.
\end{equation}
Meanwhile, $U \subset \Omega_k$ forces $\operatorname{dist}(x_k, U) \ge \|x_k - x_{k+1}\|$ and, by $x_{k+1} \in S_{ik}$, Proposition \ref{propproj} means that
$$\|x_k-x_{k+1}\|^2 \ge \| x_{k+1} - p_{i k}\|^2 + \| x_{k} - p_{i k}\|^2,$$
hence $\operatorname{dist}(x_k, U) \ge \| x_{k+1} - p_{i k}\|$. So
\begin{equation} \label{210ineq1}
\frac{\operatorname{dist}(x_{k+1}, U_i)}{\operatorname{dist}(x_k, U)} \le \frac{\operatorname{dist}(x_{k+1}, U_i)}{\| x_{k+1} - p_{i k}\|}.
\end{equation}
Now the error bound condition, implied by \(\text{int} (U) \neq \emptyset\), yields
\begin{equation} \label{210ineq2}
    \frac{\operatorname{dist}(x_{k+1}, U)}{\operatorname{dist}(x_k, U)}
\leq \frac{1}{\omega} \cdot \frac{\max_{j} \left\{\operatorname{dist}(x_{k+1}, U_j)\right\}}{\operatorname{dist}(x_k, U)} =\frac{1}{\omega} \cdot  \frac{\operatorname{dist}(x_{k+1}, U_i)}{\operatorname{dist}(x_k, U)}.
\end{equation}
Using Proposition \ref{fejerconv} and the inequality \(\text{dist}(x_k, U) \le \|x_k-\bar x\|\), we get
\begin{equation}\label{210ineq3}
    \frac{\|x_{k+1}-\bar{x}\|}{\|x_{k}-\bar{x}\|} \le 2 \cdot \frac{\operatorname{dist}(x_{k+1}, U)}{\operatorname{dist}(x_k, U)}.
\end{equation}
Combining \eqref{210lemma28}, \eqref{210ineq1}, \eqref{210ineq2} and \eqref{210ineq3}, it follows that
\[
\frac{\|x_{k+1}-\bar{x}\|}{\|x_{k}-\bar{x}\|} \le \frac{2}{\omega} \cdot \frac{\operatorname{dist}(y_k, U_i)}{\| y_k - p_{i k}\|}
\]
and, finally, from \eqref{210limit}, we arrive at
\[
\lim_{k \to \infty} \frac{\|x_{k+1}-\bar{x}\|}{\|x_{k}-\bar{x}\|} = 0,
\]
demonstrating the desired superlinear convergence.
\end{proof}

\subsection{P-CRM and MAP as special cases of 3PM}

In this section, we explore how 3PM relates to two classical projection methods, highlighting shared geometric principles and algorithmic structures. These connections help situate 3PM within the broader landscape of projection-based algorithms and provide insights into its potential strengths.

We begin by showing that 3PM recovers the Parallel Circumcentered-Reflection Method (P-CRM) \cite{barros2025parallelizing} when the second-phase projection lands exactly on the boundaries of all support half-spaces. To do so, we recall the notion of a circumcenter, which plays a central role in the formulation of CRM.

\begin{definition}\label{def:circ-m} \emph{(Circumcenter)\textbf{.}}
Let \( \{z_0, z_1, \dots, z_m\} \subset \mathbb{R}^n \). The circumcenter of this set is the point \( c \in \operatorname{aff}\{z_0, z_1, \dots, z_m\} \) such that
\[
\|c - z_0\| = \|c - z_1\| = \dots = \|c - z_m\|.
\]
\end{definition}

This point \( c \) is equidistant from all given reference points and lies within their affine span. It exists and is unique, for example, if \(z_0, z_1, \dots, z_m\) are affinely independent \cite{BauschkeOuyangWang2018}. In the context of CRM, the points \( z_i \) are obtained by reflecting the current iterate \( z_k \) across each set \( U_i \). Specifically, the Parallel CRM iteration is given by
\[
T_{\text{P-CRM}}(z) := \operatorname{circ}\left(z, R_{U_1}(z), \dots, R_{U_m}(z)\right),
\]
where each reflection is defined as \( R_{U_i}(z) := 2 P_{U_i}(z) - z \).

When the projection step in 3PM lands on the boundary of all supporting half-spaces, i.e., when \( x_{k+1} \in \partial S_{ik} \) for all \( i \), the update coincides with that of Parallel CRM. In this case, the intersection of supporting half-spaces used in 3PM aligns with the affine geometry of the reflected points in P-CRM, and both methods yield the same next iterate.

\begin{proposition} \emph{(P-CRM as a special case of 3PM)\textbf{.}}
Suppose that, at iteration \( k \), the projection onto the polyhedron \(\Omega_k\) in 3PM satisfies \( x_{k+1} \in \partial S_{ik} \) for all \( i = 1, \dots, m \); that is, all supporting half-spaces used to define \( \Omega_k \) are active at the projection step. Then, the 3PM update coincides with the P-CRM update of \( x_k \) with respect to the family \( \{U_i\}_{i=1}^m \). In other words,
\[
x_{k+1} = \operatorname{circ}\left(x_k, R_{U_1}(x_k), \dots, R_{U_m}(x_k)\right).
\]
\end{proposition}
\begin{proof}
Let us denote \( H_{ik} = \partial S_{ik}\) and  \( W_k = \text{aff}(x_k, p_{1k}, \dots, p_{mk}).\)

Since \( H:=\bigcap_i H_{ik} \) is a convex subset of \( \bigcap_i S_{ik} \) containing \( x_{k+1} \), we conclude that
\[
x_{k+1} = P_{H}(x_k).
\]
Now, since $H$ and $H_{ik}$ are subspaces with $H \subset H_{ik}$, we have
\[
x_{k+1} = P_{H}(x_k) = P_{H}(p_{i k}) \quad \text{for all } i.
\]
We deduce that for every \( w \in W_k\), we also have
\[
x_{k+1} = P_{H}(w).
\]
In particular, 
\[
x_{k+1} = P_{H}(P_{W_k}(x_{k+1})),
\]
which implies that
\begin{equation}\label{pcrmeq1}
x_{k+1} - P_{W_k}(x_{k+1}) \in H^\perp.
\end{equation}
But from \(x_k, p_{i k} \in W_k \), it follows that
\[
x_{k+1} - P_{W_k}(x_{k+1}) \in (\text{span} \{x_k - p_{i k}\})^\perp = H_{ik} \quad \forall\, i.
\]
Therefore,
\begin{equation}\label{pcrmeq2}
x_{k+1} - P_{W_k}(x_{k+1}) \in \bigcap_i H_{ik} = H.
\end{equation}
Uniting \eqref{pcrmeq1} and \eqref{pcrmeq2}, we conclude that
\[
x_{k+1} - P_{W_k}(x_{k+1}) = 0,
\]
that is
\begin{equation} \label{pcrmcond1}
    x_{k+1} \in W_k
\end{equation}
Let us now consider the reflections of \( x_k \) across each set \( U_i \), denoted by
\[
r_{i k} := 2p_{i k} - x_k.
\]
Since by assumption,
\[
\left\langle x_k - p_{i k},\, x_{k+1} - p_{i k} \right\rangle = 0,
\]
we know that the triangle with vertices \( x_k \), \( p_{i k} \), and \( x_{k+1} \) is right-angled at \( p_{i k} \).
Now, by the Pythagorean theorem, we compute the squared distance from \( x_{k+1} \) to the reflection point \( r_{i k} \):
\[
\|x_{k+1} - r_{i k}\|^2 = \|x_{k+1} - p_{i k}\|^2 + \|x_k - p_{i k}\|^2.
\]
Similarly, the distance from \( x_k \) to \( x_{k+1} \) is:
\[
\|x_k - x_{k+1}\|^2 = \|x_{k+1} - p_{i k}\|^2 + \|x_k - p_{i k}\|^2.
\]
So we conclude:
\begin{equation}\label{pcrmcond2}
\|x_{k+1} - r_{i k}\| = \|x_k - x_{k+1}\|, \quad \forall \, i.
\end{equation}
Properties \eqref{pcrmcond1} and \eqref{pcrmcond2} show that 
\[ x_{k+1} = T_{\text{P-CRM}}(x_k).\]
\end{proof}

If the condition in the proposition above holds throughout the iterations of 3PM, and assuming \(\text{int}(U) \neq \emptyset\) with each boundary \( \partial U_i \) locally a differentiable manifold, then Theorem \ref{thsplc} guarantees superlinear convergence. This behavior is reminiscent of the centralized Circumcentered-Reflection Method (cCRM) \cite{BehlingBelloCruzIusemSantos2024}, which also attains superlinear convergence under similar assumptions.

We also recover the classical Method of Alternating Projections (MAP) as a limiting case of 3PM when the number of sets is two and the iterates remain on one of the sets:

\begin{proposition}[MAP as a special case of 3PM]
If \( m = 2 \) and there exists an iteration \( k \) such that \( x_k \in U_1 \cup U_2 \), then the 3PM coincides with MAP from that iteration onward.
\end{proposition}

\begin{proof}
Suppose without loss of generality that \( x_k \in U_1 \). Then the projection \( p_{1k} = x_k \), and \( p_{2k} = P_{U_2}(x_k) \). The supporting half-spaces reduce to \( S_{1k} = \mathbb{R}^n \) and \( S_{2k} \), a half-space orthogonal to the segment from \( x_k \) to \( p_{2k} \). The polyhedral projection step in 3PM then becomes:
\[
x_{k+1} = P_{S_{2k}}(x_k) = P_{U_2}(x_k).
\]
In the next iteration, we reverse roles and project back onto \( U_1 \), so 3PM cycles between \( U_1 \) and \( U_2 \), exactly replicating MAP.
\end{proof}

These connections illustrate the versatility of 3PM: it interpolates between and generalizes important existing projection schemes while introducing a novel polyhedral correction step that can improve convergence, especially when \( m > 2 \).

\section{Approximate Parallel Polyhedral Projection Method for Convex Feasibility problems (A3PM)}\label{sec3}

\subsection{A3PM method}

In this section, we introduce the inexact extension of 3PM, called A3PM, which allows for approximate projections in both phases of the method. Exact projections are often computationally expensive or unavailable in closed form. A3PM addresses this challenge by permitting controlled inexactness in the computation of projections while still preserving convergence under mild conditions.

The algorithm starts from an arbitrary point \( x_1 = x \), and at each iteration \( k \geq 1 \), given \( x_k \), it computes approximate projections \( \widehat{P}_{U_i}(x_k, \varepsilon) \) of \( x_k \) onto each set \( U_i \), with a uniform accuracy parameter \( 0 \leq \varepsilon < 1 \). The approximate projection operator \( \widehat{P}_{U_i}(\cdot, \varepsilon) \) is defined as follows:

\begin{definition} \emph{(Approximate projection operator)\textbf{.}}
Let $X$ be a closed nonempty convex set
and let $\varepsilon \in [0,1)$.
An approximate projection ${\widehat P}_X(x,\varepsilon)$ of $x$ onto $X$ with accuracy 
$\varepsilon$
is any
point $\widehat x$
satisfying the two properties (i) and (ii) below:
\begin{itemize}
\item[(i)] \emph{dist}$(\widehat x,X) \leq \varepsilon \, \emph{dist}(x,X)$;
\item[(ii)] $X$ is contained in the set
$
\Big\{z: \langle z-\widehat x,x-\widehat x \rangle \leq 0\Big\}.
$
\end{itemize}
\end{definition}
When $\varepsilon=0$, we get the exact projection
${\widehat P}_X(x,0)=P_X(x)$
and for
$\varepsilon=0$
A3PM becomes 3PM.

The Approximate Parallel Polyhedral Projection Method (A3PM) for Convex Feasibility problems is given below.

\noindent\rule[0.5ex]{1\columnwidth}{1pt}
	
	A3PM
	
	\noindent\rule[0.5ex]{1\columnwidth}{1pt}
	\begin{itemize}
		\item [0.] Let $k=1$, $x_1=x \in \mathbb{R}^n$, $U_1, U_2, \ldots, U_m$, be closed convex sets, and take $\varepsilon \in [0,1)$.
	    \item[1.] Compute ${\widehat P}_{U_1}(x_k,\varepsilon)$, ${\widehat P}_{U_2}(x_k,\varepsilon)$, $\ldots$, ${\widehat P}_{U_m}(x_k,\varepsilon)$.
		\item[2.] Define $\widehat S_{ik}$ by
  \begin{equation}\label{sikb}
\widehat S_{ik} = \Big\{z : \left \langle x_k - {\widehat P}_{U_i}(x_k,\varepsilon),z-{\widehat P}_{U_i}(x_k,\varepsilon) \right \rangle
\leq 0 \Big\}
\end{equation}
and let 
\begin{equation}\label{omkb}
\widehat \Omega_k=\bigcap_{i=1}^m \widehat S_{ik}.
\end{equation}      
Compute 
  $$
  {\widehat P}_{\widehat \Omega_k}(x_k,\varepsilon).
  $$
		\item[3.] Do $x_{k+1} \leftarrow {\widehat P}_{\widehat \Omega_k}(x_k,\varepsilon)$, $k \leftarrow k+1$, and
go to step 1.
	\end{itemize}
	\rule[0.5ex]{1\columnwidth}{1pt}

A more general variant of A3PM allows the accuracy parameter to vary across iterations, using a sequence \((\varepsilon_k)\) in \([0,1)\) instead of a fixed value \(\varepsilon\).

\noindent\rule[0.5ex]{1\columnwidth}{1pt}
	
	A3PM with dynamic accuracy
	
	\noindent\rule[0.5ex]{1\columnwidth}{1pt}
	\begin{itemize}
		\item [0.] Let $k=1$, $x_1=x \in \mathbb{R}^n$, $U_1, U_2, \ldots, U_m$, be closed convex sets, take $(\varepsilon_k)$ a sequence or reals in  $[0,1)$.
	    \item[1.] Compute ${\widehat P}_{U_1}(x_k,\varepsilon_k)$, ${\widehat P}_{U_2}(x_k,\varepsilon_k)$, $\ldots$, ${\widehat P}_{U_m}(x_k,\varepsilon_k)$.
		\item[2.] Define $\widehat S_{ik}$ by
  \begin{equation}\label{sikb}
\widehat S_{ik} = \Big\{z : \left \langle x_k - {\widehat P}_{U_i}(x_k,\varepsilon_k),z-{\widehat P}_{U_i}(x_k,\varepsilon_k) \right \rangle
\leq 0 \Big\}
\end{equation}
and let 
\begin{equation}\label{omkb}
\widehat \Omega_k=\bigcap_i \widehat S_{ik}.
\end{equation}      
Compute 
  $$
  {\widehat P}_{\widehat \Omega_k}(x_k,\varepsilon_k).
  $$
		\item[3.] Do $x_{k+1} \leftarrow {\widehat P}_{\widehat \Omega_k}(x_k,\varepsilon_k)$, $k \leftarrow k+1$, and
go to step 1.
	\end{itemize}
	\rule[0.5ex]{1\columnwidth}{1pt}

Figure \ref{fig:A3PM} illustrates the computation of an iteration of A3PM for 3 sets $U_1$, $U_2$, and $U_3$
in $\mathbb{R}^2$. In this figure we represent:
\begin{itemize}
\item convex sets $U_1$, $U_2$, and $U_3$ and their intersection $U$;
\item an arbitrary initial point $x$;
\item the approximate projections ${\widehat P}_{U_1}(x,\varepsilon)$, ${\widehat P}_{U_2}(x,\varepsilon)$, and
${\widehat P}_{U_3}(x,\varepsilon)$ of $x$ onto respectively $U_1$, $U_2$, and
$U_3$;
\item half-spaces $\widehat S_1$, $\widehat S_2$, and $\widehat S_3$ which play the role of half-spaces $\widehat S_{ik}$ in A3PM; we dropped the iteration index
for simplicity;
\item the intersection $\widehat \Omega=\widehat S_1 \cap \widehat S_2 \cap \widehat S_3$ which plays the role of $\widehat \Omega_k$ in A3PM;
\item the next iterate which is the approximate projection 
${\widehat P}_{\widehat \Omega}(x,\varepsilon)$ of $x$ onto $\widehat \Omega$.
\end{itemize}

\begin{figure}
    \centering
    \includegraphics[width=0.95\textwidth]{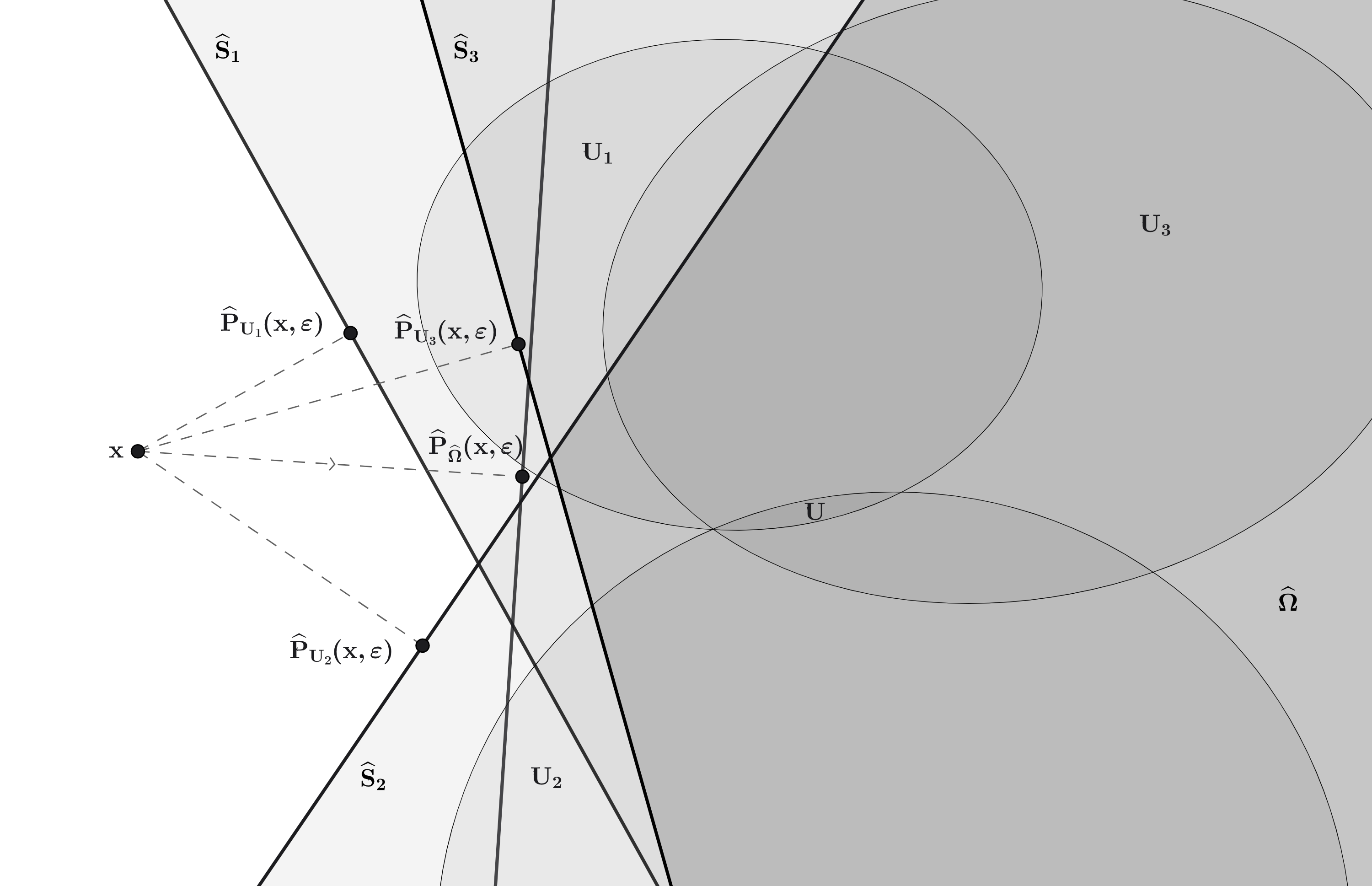}
    \caption{An iteration of A3PM}
    \label{fig:A3PM}
\end{figure}

\subsection{Convergence analysis}

In this section, we study the convergence properties of A3PM in its most general form, where the accuracy level of the approximate projections may vary across iterations. That is, we consider a sequence of accuracies \( (\varepsilon_k) \subset [0,1) \) used in each step of the algorithm. We assume throughout this section that
\[
\sup_k \varepsilon_k < 1.
\]
This dynamic formulation includes, as a special case, the fixed-accuracy variant of A3PM where a single value \( \varepsilon \in [0,1) \) is used at every iteration. Thus, the results below apply to both forms of the algorithm.

We begin with an elementary property of the approximate projection:

\begin{lemma}\label{lemprapp} \emph{(Bound on step size of approximate projection)\textbf{.}}
    For any set $X$ and point $x$, let ${\widehat P}_X(x,\varepsilon)$ be an approximate projection of $x$ onto $X$. Then 
\begin{equation}
    \|x-{\widehat P}_{X}(x,\varepsilon)\| \ge (1-\varepsilon) \, \emph{dist}(x, X).
\end{equation}
\end{lemma}
\begin{proof} Setting
$\widehat x={\widehat P}_{X}(x,\varepsilon)$,
by definition of the approximate projection,
    \[\text{dist}(\widehat x, X) = \|\widehat x - P_{X}(\widehat x)\| \le \varepsilon \, \text{dist}(x, X).\]
    Hence,
    \begin{align*}
        \|x-\widehat x\| &\ge \|x-P_{X}(\widehat x)\| - \|\widehat x - P_{X}(\widehat x)\|
        \\ &\ge \text{dist}(x, X)- \varepsilon \, \text{dist}(x, X)
        \\ &= (1-\varepsilon) \, \text{dist}(x, X).
    \end{align*}
\end{proof}

The following lemma is for A3PM the analogue of Lemma
\ref{lemmadist}
for 3PM:

\begin{lemma}\label{lemmadistb} \emph{(Vanishing setwise distance for A3PM)\textbf{.}} For any starting point $x_1 \in \mathbb{R}^n$, the sequence $(x_k)$ generated by A3PM satisfies
\begin{equation}\label{first}
\displaystyle \lim_{k \rightarrow +\infty} 
 \max_{i=1,\ldots,m} \emph{dist}(x_k,U_i) =0.
\end{equation}
\end{lemma}
\begin{proof} 
Lemma \ref{lemprapp} with $x=x_k$, $\varepsilon=\varepsilon_k$, and
$X=\widehat \Omega_k$, gives
\begin{equation}\label{firstapp}
    \|x_k-x_{k+1} \| \ge (1-\varepsilon_k) \, \text{dist}(x_k, \widehat \Omega_k).
\end{equation}
Next, notice that 
\begin{equation*}
    \widehat \Omega_k \subset \widehat S_{ik} \implies \text{dist}(x_k, \widehat \Omega_k) \ge \text{dist}(x_k, \widehat S_{ik}) = \|x_k - \widehat P_{U_i}(x_k, \varepsilon_k)\|
\end{equation*}
and using Lemma \ref{lemprapp} with $x=x_k$, $\varepsilon=\varepsilon_k$, and
$X=U_i$, we deduce for all $i$:
\begin{equation}\label{secprox}
    \text{dist}(x_k, \widehat \Omega_k) \ge (1-\varepsilon_k) \, \text{dist}(x_k,U_i).
\end{equation}
Relations \eqref{firstapp} and \eqref{secprox} give
\begin{equation}
    \|x_k-x_{k+1} \| \ge (1-\varepsilon_k)^2 \, \displaystyle \max_{i=1,\ldots,m}\text{dist}(x_k,U_i).
\end{equation}
We have, by definition of $x_{k+1}$ and of the approximate projection, that for all $u \in \widehat \Omega_k$
\begin{equation}\label{secapph}
\langle u-x_{k+1}, x_k-x_{k+1} \rangle \le 0
\end{equation}
and since $U_i \subset \widehat S_{ik}$, $U \subset \widehat \Omega_k$, we also have for every $u \in U$ that
\eqref{secapph} holds. It follows that for every $u \in U$,
\begin{equation}\label{conveqapprox}
    \begin{array}{lcl}
        \|u-x_k\|^2 & = & \|x_k-x_{k+1}\|^2 + \|u-x_{k+1}\|^2
        +2\langle u-x_{k+1},x_{k+1}-x_k\rangle
        \\ & \ge & \displaystyle (1-\varepsilon_k)^4 \, \max_{i=1,\ldots,m} \left \{ \text{dist}(x_k,U_i)^2 \right\} + \|u-x_{k+1}\|^2.
        \\ & \ge & \displaystyle (1-\sup \varepsilon_k)^4 \, \max_{i=1,\ldots,m} \left \{ \text{dist}(x_k,U_i)^2 \right\} + \|u-x_{k+1}\|^2.
\end{array}
\end{equation}
We conclude that $(\|u-x_{k}\|)_k$ is decreasing and nonnegative, so it is convergent. Then, since $\sup \varepsilon_k < 1,$ we obtain 
\eqref{first}.
\end{proof}

\if{
Notice that 
\begin{equation*}
    \widehat \Omega_k \subset \widehat S_{ik} \implies \text{dist}(x_k, \widehat \Omega_k) \ge \text{dist}(x_k, \widehat S_{ik}) = \|x_k - \widehat P_{U_i}(x_k, \varepsilon_k)\|
\end{equation*}

so, again by the lemma,

\begin{equation}
    \text{dist}(x_k, \widehat \Omega_k) \ge (1-\varepsilon_k) \, \text{dist}(x_k,U_i).
\end{equation}

Hence 

\begin{equation}
    \|x_k-x_{k+1} \| \ge (1-\varepsilon_k)^2 \, \text{dist}(x_k,U_i).
\end{equation}

Given any \(u \in U,\) from \(\langle u-x_{k+1}, x_k-x_{k+1} \rangle \le 0\) we get
\begin{align}
        \|u-x_k\|^2 & \ge \|x_k-x_{k+1}\|^2 + \|u-x_{k+1}\|^2
        \\ & \ge (1-\varepsilon_k)^4 \, \max_{i} \left \{ \text{dist}(x_k,U_i)^2 \right\} + \|u-x_{k+1}\|^2.
\end{align}

We conclude that $(\|u-x_{k}\|)_k$ is decreasing and nonnegative, so it is convergent. 

Then if $\sup \varepsilon_k < 1,$ we get
\begin{equation}
    \max_{i} \left \{ \text{dist}(x_k,U_i) \right\} \to 0
\end{equation}
}\fi

The following theorem proves the convergence of A3PM.
\begin{theorem}\label{convA3PM} \emph{(Global convergence of A3PM)\textbf{.}} For any starting point $x_1 \in \mathbb{R}^n$, the sequence $(x_k)$ generated by A3PM converges to a point 
$\bar x \in U$.
\end{theorem}
\begin{proof} Recall that Lemma \ref{lemmadistb} for A3PM
is the analogue of Lemma \ref{lemmadist} for 3PM.
Also observe that Lemma \ref{bdedappproj} still holds
for A3PM since the relation $\|x_{k+1}-u\| \leq \|x_k-u\|$
for all $u \in U$ on which it is based also
holds. We can then prove Theorem \ref{convA3PM} similarly to 
Theorem \ref{conv3PM}.
\end{proof}

\if{
Also, $(x_k)$ is bounded. 

Taking any convergent subsequence \(x_{n_k} \to \bar{x}\), by continuity of distance, 
\begin{equation}
\max_{i} \left \{ \text{dist}(\bar{x},U_i) \right\} = 0 \implies \bar{x} \in U.
\end{equation}

Now taking $u=\bar{x}$, we know that $(\|\bar{x}-x_{k}\|)_k$ is convergent, then the limit should be $0.$

}\fi

We conclude the analysis of A3PM by showing that linear convergence is also retained under the error bound condition.

\begin{theorem}\label{linconvA3PM} \emph{(Linear convergence of A3PM)\textbf{.}} Let $(x_k)$ be the sequence generated by A3PM for a starting point $x_1 \in \mathbb{R}^n$. Assume that \ref{EB} holds for $ \{U_i\}_{i=1}^m $ around the limit point $\bar x$ of $(x_k)$. Then $(x_k)$ converges linearly to $\bar x$.
\end{theorem}

\begin{proof} By \ref{EB}, for some $\omega \in (0,1)$, we have
\begin{equation}
    \displaystyle \omega \, \text{dist}(x_k,U) \leq \max_{i=1,\ldots,m} \text{dist}(x_k,U_i)
\end{equation}
for every sufficiently large $k$. Notice that $(x_k)$ is Fejér monotone with respect to $U$ by \eqref{conveqapprox}. Hence we can apply Proposition \ref{fejerconv} to conclude that
\begin{equation}\label{rlindist}
    \displaystyle \text{dist}(x_k,U) \geq \frac{1}{2}\|x_k - \bar{x}\|
\end{equation}
for any $k$. Now from \eqref{conveqapprox} again with $u = \bar{x}$, 
\begin{equation}\label{linconveq1}
    \begin{array}{lcl} 
    \|x_{k}-\bar{x}\|^2 & \geq & \displaystyle (1-\varepsilon_k)^4 \, \max_{i=1,\ldots,m} \left \{ \text{dist}(x_k,U_i)^2 \right\} + \|x_{k+1}-\bar{x}\|^2
    \\ & \geq & (1-\varepsilon_k)^4 \, \omega^2 \, \text{dist}(x_k,U)^2 + \|x_{k+1}-\bar{x}\|^2
    \\ & \geq & \displaystyle (1-\varepsilon_k)^4 \, \frac{\omega^2}{4}\|x_k - \bar{x}\|^2 + \|x_{k+1}-\bar{x}\|^2
    \\ & \geq & \displaystyle (1-\sup \varepsilon_k)^4 \, \frac{\omega^2}{4}\|x_k - \bar{x}\|^2 + \|x_{k+1}-\bar{x}\|^2.
\end{array}
\end{equation}
Rearranging the terms, we arrive at
\begin{equation}\label{linconveq2}
\displaystyle \sqrt{1-(1- \sup \varepsilon_k)^4 \,\frac{\omega^2}{4}} \,\|x_{k}-\bar{x}\| \geq \|x_{k+1}-\bar{x}\|,
\end{equation}
establishing the linear convergence of A3PM.

\end{proof}

\section{Numerical experiments}\label{sec4}

To see how well the proposed 3PM and A3PM methods perform in practice,
we consider the problem of finding a point in the intersection of
$m$ ellipsoids $U_1,$ $\ldots,$ $U_m$ in $\mathbb{R}^n$ where ellipsoid $U_i$ is centered
at $y_i \in \mathbb{R}^n$ and is given by
\begin{equation}
U_i=\{y \in \mathbb{R}^n: (y-y_i)^T Q_i (y-y_i) \leq \eta_i^2\}
\end{equation}
for some positive $\eta_i$ and definite positive matrices $Q_i$.
We generate the centers
$y_i$ randomly 
and $Q_i$ of the form
$Q_i=A_i A_i^T + \lambda_i I_n$ where
matrices $A_i$ and positive $\lambda_i$ are generated randomly too. 
In that manner, matrices $Q_i$ are definite positive, as desired. We also choose $\eta_i$ sufficiently large so that the intersection of sets $U_i$ contains
the unit ball
$B=\{x: \|x\|_2 \leq 1\}$, namely
\begin{equation}\label{eqki}
\eta_i \geq (1+\|y_i\|_2) \sqrt{\|Q_i\|_2}
\end{equation}
for all $i$ (in particular, the intersection is nonempty). Indeed if \eqref{eqki} is satisfied and $x$ belongs to $B$ then
$$
(x-y_i)^T Q_i (x-y_i) \leq \|x-y_i\|_2^2 \|Q_i\|_2 \leq (1+\|y_i\|_2)^2 \|Q_i\|_2 \leq 
\eta_i^2
$$
and therefore $x \in U_i$. Initial points $x_0$ are also
generated randomly outside the ellipsoids.

For convenience, in what follows, for
$i=1,\ldots,m$,
we define
function $g_i(y)=(y-y_i)^T Q_i (y-y_i)-\eta_i^2$ which is differentiable with gradient
$g_i'(y)=2Q_i(y-y_i)$.
We compare 3PM and A3PM with four competing methods: cyclic projections,
Cimmino's method, Successive Centralized CRM (SCCRM) from
\cite{BehlingBelloCruzIusemSantos2024}, and CRM in the product space (denoted by CRM in what follows) from \cite{crmprod}. 

We now briefly describe the computations
of the iterations of these methods. 

For the details of A3PM, we will use the
approximate projection proposed by Fukushima \cite{Fukushima1983} that we start recalling.
More precisely, the following proposition defines an approximate projector over a closed convex set, which, without loss of generality, can be written under the form
$\{x \in \mathbb{R}^n : g(x) \leq 0\}$ for
$g:\mathbb{R}^n \rightarrow \mathbb{R}$.

\begin{proposition}\label{apprproj}
Consider the
 closed convex set \(K := \{ x \mid g(x) \leq 0 \}\) where \(g(x): \mathbb{R}^n \rightarrow \mathbb{R} \) is convex.
Suppose 
that there exists
$\hat x$ such that
$ g(\hat{x}) < 0$ (Slater condition).  Let \( g^\prime \) satisfy \( g^\prime(x) \in \partial g(x) \) for any \( x \).  
Moreover, let  
\[
\tilde{p}(x) := \underset{y}{\operatorname{argmin}} \|y - x \| \quad \text{s.t.} \quad g(x) + g^\prime(x)^\top (y - x) \leq 0.
\]
Then
 $$
  \displaystyle \tilde{p}(x) =
    \begin{cases}
x, & \text{if } g(x) \leq 0\\ 
    x - \dfrac{g(x)}{\| g^\prime(x) \|^2} g^\prime(x), & \text{otherwise},
    \end{cases}
$$
and $\tilde{p}$  
is an
$\varepsilon$-approximate projector over $K$ for some 
$0<\varepsilon<1$.
\end{proposition}

We are now ready to provide the details of the computation of all methods.\\

\par {\textbf{3PM.}} At iteration $k$, given $x_k$, 3PM computes 
\begin{equation}
p_{i k}=\left\{
\begin{array}{l}
\displaystyle \underset{y \in \mathbb{R}^n}{\argmin} \, \frac{1}{2}\|y-x_k\|_2^2 \\ 
g_i(y) \leq 0,
\end{array}
\right.
\end{equation}
for $i=1,\ldots,m$
and then sets
\begin{equation}
x_{k+1} = \left\{
\begin{array}{l}
\displaystyle \underset{y \in \mathbb{R}^n}{\argmin} \, \frac{1}{2}\|y-x_k\|_2^2 \\ 
(x_k-p_{i k})^T (y-p_{i k}) \leq 0,\;i=1,\ldots,m.
\end{array}
\right.
\end{equation}

\par {\textbf{A3PM.}} 
We first compute the approximate projections over the ellipsoids using Proposition \ref{apprproj}.
At iteration $k$, for $i=1,\ldots,m$, if
$g_i(x_k)\leq 0$, we compute
$$
p_{ik}=x_k,
$$
otherwise, we compute
$$
p_{ik}=x_k-\frac{g_i(x_k)}{\|g_i'(x_k)\|_2^2} g_i'(x_k).
$$
For the final step of A3PM, we look for an
approximate projection of $x_k$
onto the closed convex set 
$K=\{x: h(x) \leq 0\}$ where
$h(x)=\max_{i=1,\ldots,m} h_i(x)$ for
$$
h_i(x)=\langle x_k-p_{ik},x-p_{ik}\rangle.
$$
A subgradient $h'(x_k)$ of $h$ at $x_k$ is given by 
$$
h'(x_k)=x_k-p_{i_k k}
$$
where index $i_k$ is any index $i$ satisfying $h(x_k)=h_i(x_k)$.

With this notation, if $h(x_{k}) \leq 0$ then the approximate projection is $x_{k+1}=x_k$ otherwise it is given by
$$
x_{k+1}=x_k-\frac{h(x_k)}{\|h'(x_k)\|_2^2}h'(x_k).
$$
Though the letter P in 3PM and A3PM stands for Parallel and refers to the fact that 
the projections can be computed in parallel in these methods, we can also implement nonparallel variants
of these methods where projections are computed
sequentially.
In what follows, we denote
by 3PM and A3PM the nonparallel
implementations while the parallel implementations are denoted by 3PM $\parallel$ and
A3PM $\parallel$, respectively.\\

\par {\textbf{Cyclic projections.}} For iteration $k$, given $x_k$, set
$x=x_k$ and for $i=1,\ldots,m$,  do iteratively
\begin{equation}
x \leftarrow \left\{
\begin{array}{l}
\displaystyle \underset{y \in \mathbb{R}^n}{\argmin} \, \frac{1}{2}\|y-x\|_2^2 \\ 
g_i(y) \leq 0.
\end{array}
\right.
\end{equation}
After these $m$ updates (projections), do $x_{k+1} \leftarrow x$,
$k \leftarrow k+1$.\\

\par {\textbf{Cimmino's method.}} At iteration $k$, do
$$
x_{k+1} \leftarrow \frac{1}{m} \sum_{i=1}^m \left\{
\begin{array}{l}
\displaystyle \underset{y \in \mathbb{R}^n}{\argmin} \, \frac{1}{2}\|y-x_k\|_2^2 \\ 
g_i(y) \leq 0,
\end{array}
\right.
$$
and set $k \leftarrow k+1$. For Cimmino's method, we can also consider a parallel implementation which computes in parallel the $m$ projections. We denote by Cimmino $\parallel$ the corresponding implementation.\\

\par {\textbf{Successive Centralized CRM (SCCRM).}} We now recall the SCCRM method
described in \cite{succccrm}.
We first define the alternating projection operator
$Z_{A,B}:\mathbb{R}^n \rightarrow \mathbb{R}^n$
given by
$$
Z_{A,B}=P_A \circ P_B
$$
as well as the
simultaneous projection operator
$\tilde Z_{A,B}:\mathbb{R}^n \rightarrow \mathbb{R}^n$ given by
$$
\tilde Z_{A,B}=\frac{1}{2}[P_A+P_B].
$$
We also define the operator
$\bar Z_{A,B}:\mathbb{R}^n \rightarrow \mathbb{R}^n$ by
$$
\bar Z_{A,B}=\tilde Z_{A,B} \circ Z_{A,B}
$$
and the operator
$$
\mathcal{C}_{A,B}(z)=\mbox{circ}(x,R_A(x),R_B(x)).
$$
where 
$R_A=2P_A-I$, $R_B=2P_B-I$, $\mbox{circ}(a,b,c)$ is the circumcenter of
points $a,b,$ and $c$.
With this notation, we finally define the operator $T_{A,B}:\mathbb{R}^n \rightarrow \mathbb{R}^n$
by
$$
T_{A,B}(x)=\mathcal{C}_{A,B}(\bar Z_{A,B}(x)).
$$
The
SCCRM method for the convex feasibility problem computes
$$
x_{k+1}=T_{U_{r(k)} U_{\ell(k)}}(x_k)
$$
where
the sequences $\{\ell(k)\},\{r(k)\}$ determining which sets are used at 
$k$th iteration are called control sequences. 
 We use the following control sequence $\ell(k) = 1,2,3,\ldots,m-1,m,1,2,\ldots$, and $r(k)=2,3,4,\ldots,m-1,m,1,2,3,\ldots$ To complete the description of SCCRM, we now explain how to compute the circumcenter of a set of points.
Given \( m + 1 \) points \( x_0, x_1, x_2, \ldots, x_m \in \mathbb R^n \), their circumcenter, denoted by \( \text{circ}(x_0, x_1, x_2, \ldots, x_m) \), is defined as the unique point that lies within the affine subspace spanned by these points and that is equidistant from each of them.
By this definition, we must have, for some $\alpha \in \mathbb R^m,$
\[\text{circ}(x_0, x_1, x_2, \ldots, x_m) = x_0 + \sum_{j=1}^m \alpha_j (x_j - x_0),\]
and, for all $i,$
\[
\|\text{circ}(x_0, x_1, x_2, \ldots, x_m)-x_i\| = \|\text{circ}(x_0, x_1, x_2, \ldots, x_m)-x_0\|.
\]
These conditions lead to an \( m \times m \) linear system in the coefficients \( \alpha \in \mathbb R^m \), where the \( i \)-th equation is
\[
\sum_{j=1}^m \alpha_j \langle x_j - x_0, x_i - x_0 \rangle = \frac{1}{2} \| x_i - x_0 \|^2,\]
or
\[
\begin{pmatrix}
\langle x_1 - x_0, x_1 - x_0 \rangle & \langle x_2 - x_0, x_1 - x_0 \rangle & \cdots & \langle x_m - x_0, x_1 - x_0 \rangle \\
\langle x_1 - x_0, x_2 - x_0 \rangle & \langle x_2 - x_0, x_2 - x_0\rangle & \cdots & \langle x_m - x_0, x_2 - x_0 \rangle \\
\vdots & \vdots & \ddots & \vdots \\
\langle x_1 - x_0, x_m - x_0 \rangle & \langle x_2 - x_0, x_m - x_0 \rangle & \cdots & \langle x_m - x_0, x_m - x_0 \rangle
\end{pmatrix}
\begin{pmatrix}
\alpha_1 \\
\alpha_2 \\
\vdots \\
\alpha_m
\end{pmatrix}
= \frac{1}{2}
\begin{pmatrix}
\| x_1 - x_0 \|^2 \\
\| x_2 - x_0 \|^2 \\
\vdots \\
\| x_m - x_0 \|^2
\end{pmatrix}.
\]
Solving this system determines the circumcenter \(\text{circ}(x_0, x_1, x_2, \ldots, x_m)\).\\

\par {\textbf{CRM in product space.}}
Take $x_0  \in \mathbb{R}^n$
and define
$z_0=(x_0,x_0,\ldots,x_0) \in \mathbb{R}^{n m}$.
Let 
$W=U_1 \times U_2  \times \ldots \times U_m$ and
$D=\{(x,x,\ldots,x) \in \mathbb{R}^{nm}: x \in \mathbb{R}^n\}$.
For $z=(x^{(1)},x^{(2)},\ldots,x^{(m)}) \in \mathbb{R}^{nm}$, the projection
$P_D(z)$ of $z$ onto $D$ is given by
$$
P_D(z)=\frac{1}{m}\left(\sum_{i=1}^m x^{(i)},\sum_{i=1}^m x^{(i)},\ldots,\sum_{i=1}^m x^{(i)} \right)
$$
while the projection $P_W(z)$ of $z$
onto $W$ is given by
$$
P_W(z)=\left(P_{U_1}(x^{(1)}),P_{U_2}(x^{(2)}),\ldots,P_{U_m}(x^{(m)})\right).
$$
The iterations of CRM in the product space $W \times D$ are given by
$$
z^{k+1}=\mbox{circ}(z^k,R_W(z^k),R_D(R_W(z^k)))
$$
where $R_W=2P_W-I$ and 
$R_D=2P_D-I$.

We say that a point $y$ is an 
$\varepsilon$-approximate solution of the problem
if $(y-y_i)^T Q_i (y-y_i) \leq (\eta_i+\varepsilon)^2$
for all $i$. All methods are run until we find
an $\varepsilon$-approximate solution or stop
if after 10 minutes no $\varepsilon$-approximate solution was found. We call
constraints violation at termination
the quantity 
$$
\displaystyle \max_{i=1,\ldots,m} (y-y_i)^T Q_i (y-y_i)-(\eta_i+\varepsilon)^2,
$$
which should be nonpositive
at termination, unless the method was run for at least 10 minutes
without finding an $\varepsilon$-approximate solution.

We first run the methods for $m=3$ ellipsoids
in $\mathbb{R}^2$. For this setup, we plot in Figure 
\ref{fig:QCC} these three ellipsoids and represent the evolution of the iterates for all methods.
As expected, all methods stop after finding a point
in the intersection of the ellipsoids or very close to this intersection. The plot also represents the unit ball for the $\|\cdot\|_2$-norm which, as we recall, is contained in all ellipsoids.

\begin{figure}
    \centering
    \begin{tabular}{cc}
\includegraphics[scale=0.4]{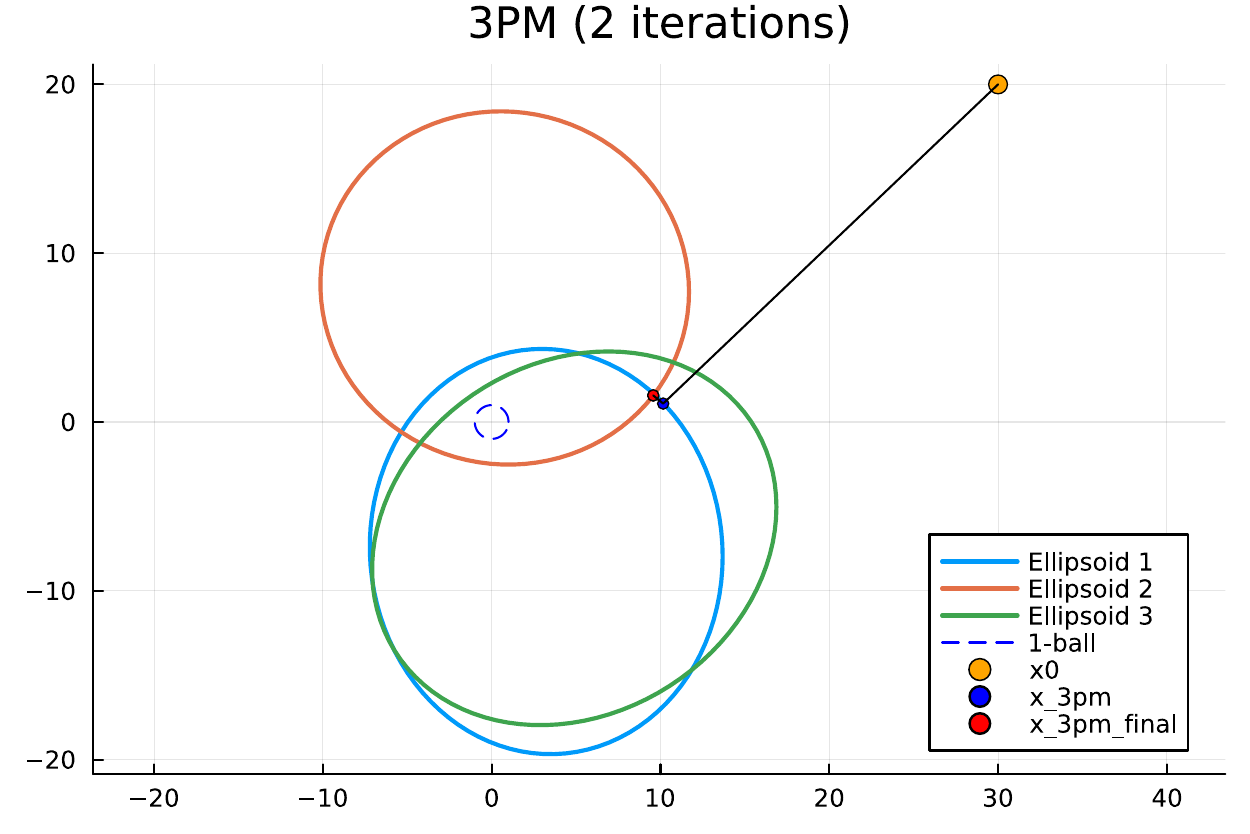}
¨& \includegraphics[scale=0.4]{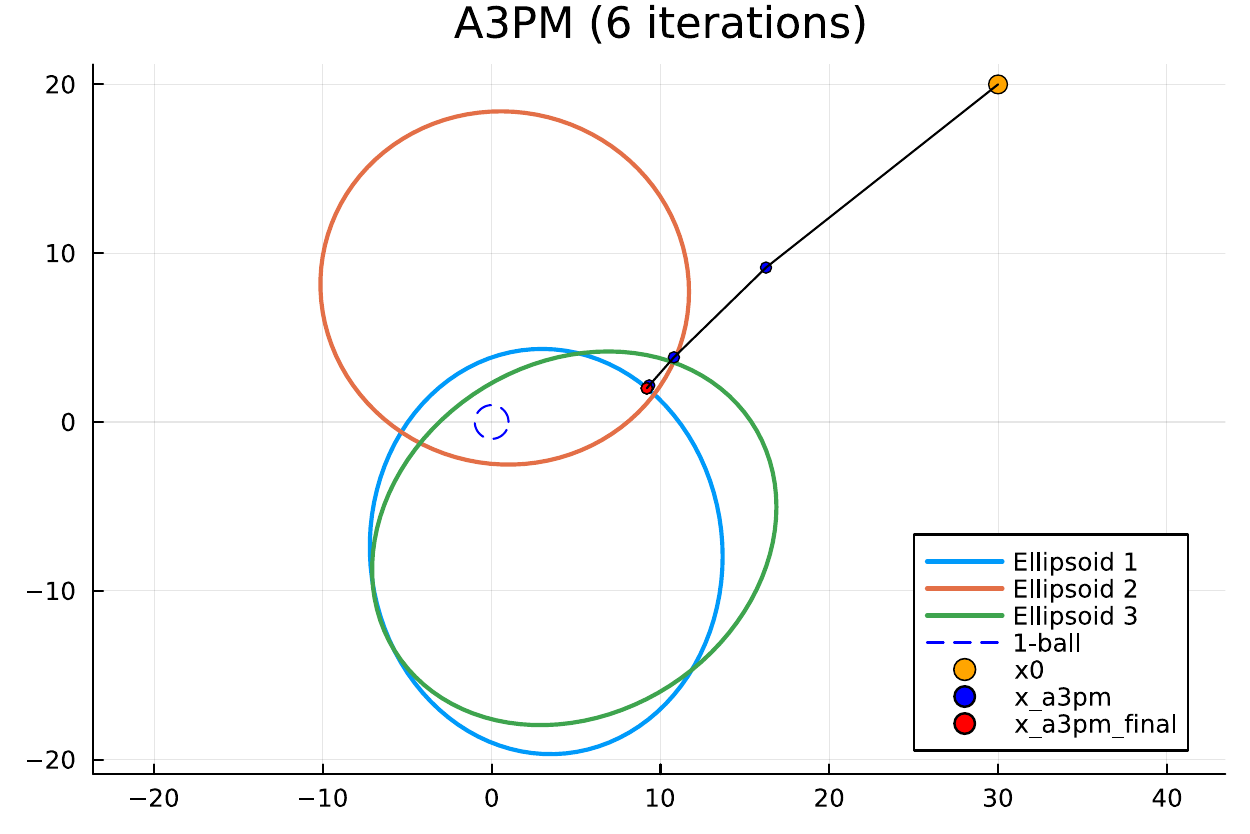}\\
\includegraphics[scale=0.4]{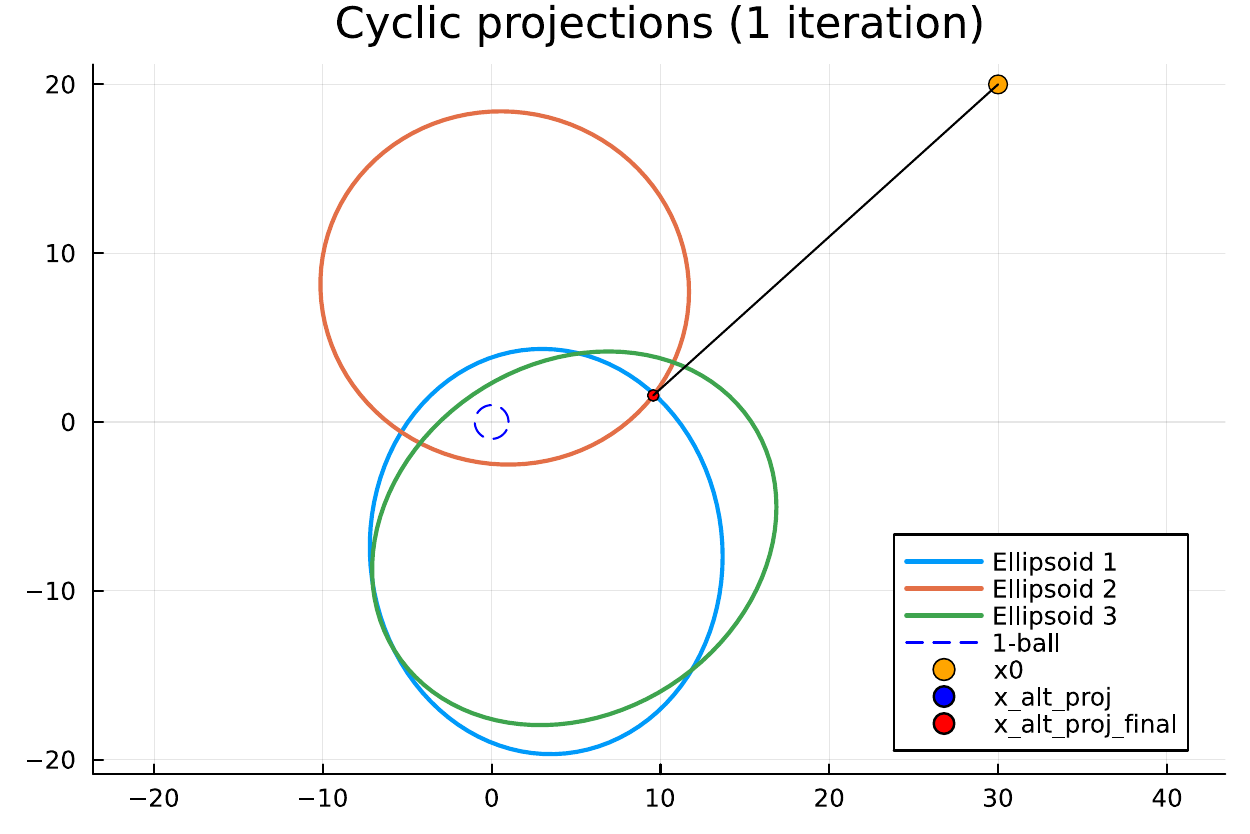}& \includegraphics[scale=0.4]{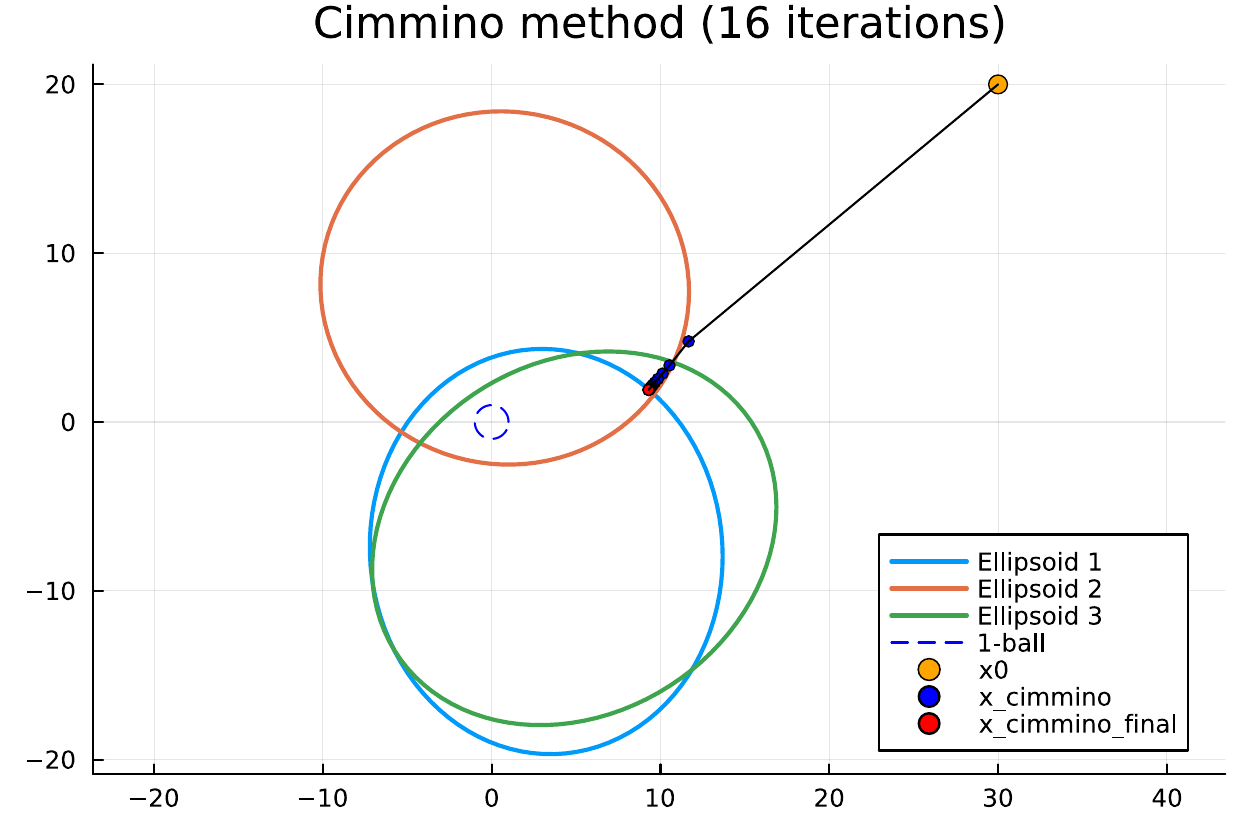}\\
\includegraphics[scale=0.4]{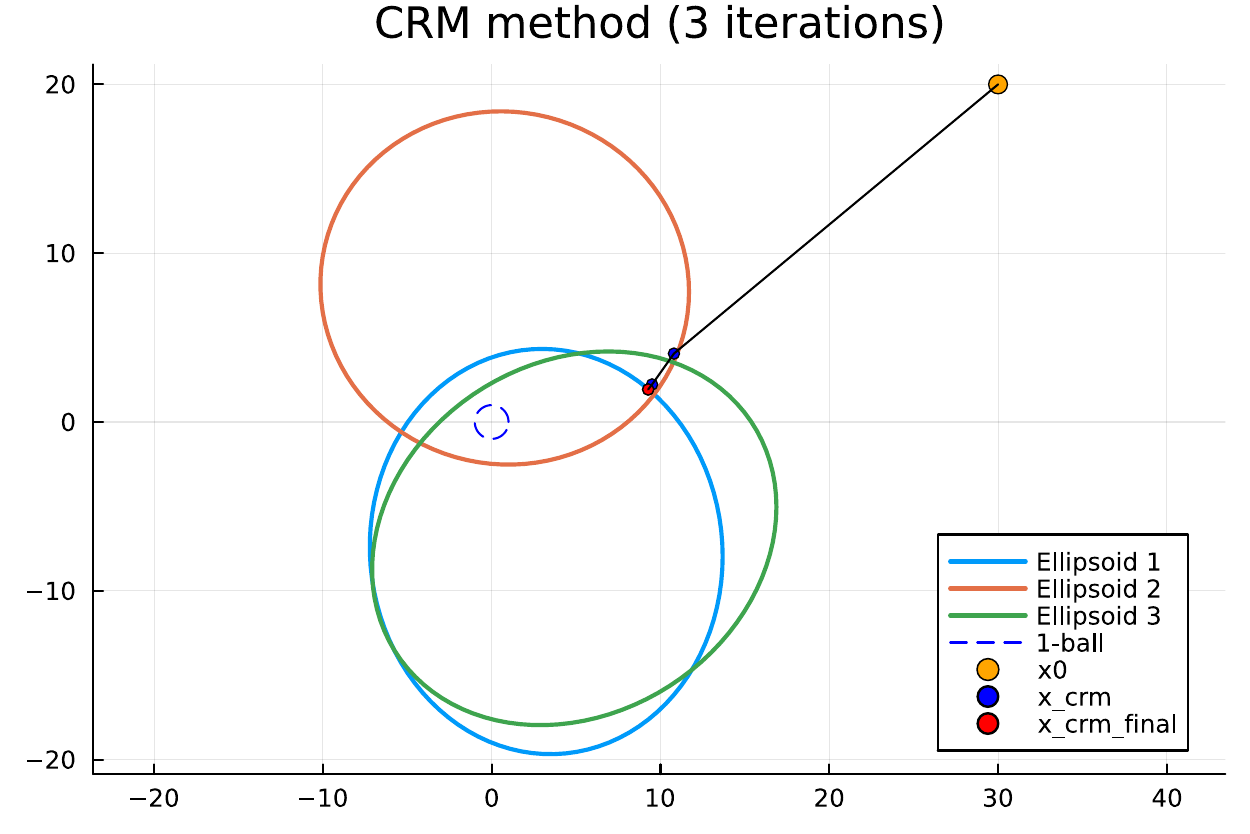}& \includegraphics[scale=0.4]{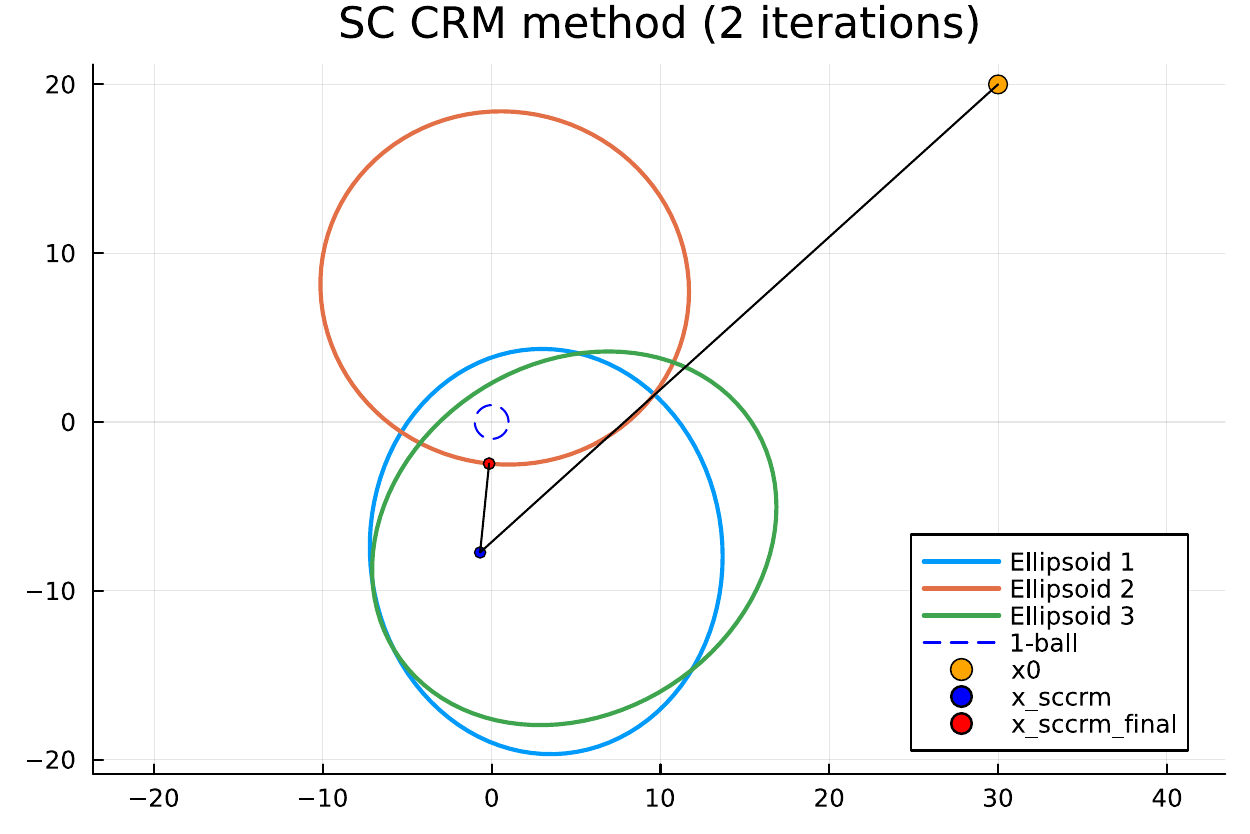}
\end{tabular}
    \caption{Illustration of 3PM, A3PM, cyclic projections, Cimmino's method, CRM, and SCCRM methods to find a point in the intersection of 3 ellipses in $\mathbb{R}^2$.}
    \label{fig:QCC}
\end{figure}

We now run the methods on 10 instances of our convex feasibility problem for
$(m,n)=(3,10)$, $(m,n)=(3,50)$, 
$(m,n)=(3,100)$, $(m,n)=(3,1000)$, 
$(m,n)=(10,100)$, $(m,n)=(10,500)$, 
$(m,n)=(10,1000)$, $(m,n)=(50,500)$, 
$(m,n)=(50,1000)$, $(m,n)=(100,1000)$. For this experiment, we report in Tables
\ref{tableres} and \ref{tableres1}, 
for the nine
methods described before, namely 3PM, 3PM $\parallel$,
A3PM, A3PM $\parallel$, Cyclic projections, Cimmino, Cimmino $\parallel$, SCCRM, and CRM, the number of iterations, the CPU time in seconds, and the constraints violation at termination. In these tables, we highlight
in bold and in red the quickest method for all problem instance. Some comments are in order:
\begin{itemize}
\item A3PM or A3PM $\parallel$ are the best methods in terms of CPU time, in 8 of the 10 instances.
\item In terms of CPU time, SCCRM is the best in two instances but very close to A3PM in one of these two.
\item Parallel implementations are in general quicker than their nonparallelized counterpart, but not always, due to time required to distribute the computations in different threads.
\item Cimmino and  CRM are  the worst
methods in terms of CPU time in these instances.
\item Though we have shown superlinear convergence of A3PM under some assumptions, in our instances A3PM did not perform well for large instances.
\end{itemize}
We also report in Figure \ref{fig:QCC2} the evolution of the
constraints violations along iterations for all methods and 5 of the instances.  We observe that Cimmino and A3PM tend to use more
iterations and the constraints violation decreases slower with these methods.
SCCRM needs in general few iterations, though it can sometimes require
a significant number of iterations before finding an
$\varepsilon$-approximate solution, see the instance
with $m=50$, $n=100$.

\begin{table}
\centering
\begin{tabular}{|c|c|c|c|c|c|}
 \hline
 Method & m &  n & Iterations &  CPU time (s) & Constraints violation\\
 \hline
 3PM&3 &10& 1 & 0.00337  &  -0.0025\\
 \hline
 3PM $\parallel$ & 3& 10&1 &  0.00259	& -0.0025   \\
 \hline
 {\textbf{\color{red}A3PM}} & 3&10&7 &   {\textbf{\color{red}0.00032}} 	& -0.7877  \\
 \hline
 A3PM $\parallel$ & 3&10&7 &   0.00070 &	-0.7877  \\
 \hline
Cyclic  & 3&10&1 & 0.00218	& -0.0025\\
 \hline
Cimmino  & 3&10& 48& 0.05752 & 	-7.5734\\
 \hline
Cimmino $\parallel$ & 3&10&48 & 0.05009	& -7.5734\\
 \hline
SCCRM  & 3&10&2 & 0.01476 & -0.0217\\
 \hline
 CRM & 3&10& 2& 0.00688 & -0.1923 \\
 \hline
 \hline 
   3PM& 3&50 &3& 0.05318  &	-56.3220 \\
 \hline
 3PM $\parallel$ & 3&50& 3 & 0.03118 & -56.3220\\
 \hline
 A3PM & 3&50& 8& 0.00098	&-167.1852 \\
 \hline
 {\textbf{\color{red}A3PM $\parallel$}} & 3&50&8 & {\textbf{\color{red}0.00078}}	&-167.1852\\
 \hline
Cyclic  & 3&50& 1 & 0.01374	&-0.2105\\
 \hline
Cimmino  & 3&50& 1& 0.01284&-415999.5440\\
 \hline
Cimmino $\parallel$ & 3&50 & 1 & 0.00707	&-415999.5440\\
 \hline
SCCRM  & 3&50& 1&0.02615	& -89742.6906\\
 \hline
 CRM & 3&50& 4 & 0.08202	&-1536.7657\\
 \hline
 \hline
 3PM& 3 &100 & 3& 0.14799 & 	-2466.1367  \\
 \hline
 3PM $\parallel$ & 3&100& 3& 0.08301	&-2466.1366 \\
 \hline
 A3PM & 3&100& 8& 0.00806	&-3386.0654\\
 \hline
 {\textbf{\color{red}A3PM $\parallel$}}  &3 &100&8 & {\textbf{\color{red}0.00170}}	&-3386.0654 \\
 \hline
Cyclic  & 3&100& 1& 0.04907	&-2.2873 \\
 \hline
Cimmino  & 3&100& 18&0.42693	&-228.6494\\
 \hline
Cimmino $\parallel$ &3 &100& 18& 0.47930&	-228.6494\\
 \hline
SCCRM  & 3&100&  2&0.10322	&-41251.7301 \\
 \hline
 CRM &3 &100&  3&0.25536	&-19558.2516 \\
 \hline
 \hline
  3PM& 3&1000&3 & 12.4355 & -1.41e7 \\
 \hline
 3PM $\parallel$ &3 &1000&3 & 9.2820	& -1.41e7\\
 \hline
  {\textbf{\color{red}A3PM}} & 3&1000& 5&  {\textbf{\color{red}0.1971}} & -1.18e7\\
 \hline
 A3PM $\parallel$ &3 &1000&5 & 0.3847  &	-1.18e7\\
 \hline
Cyclic  & 3&1000& 1& 5.3111 & -1.15e6\\
 \hline
Cimmino  & 3&1000&1 & 5.4527  &	-1.21e9\\
 \hline
Cimmino $\parallel$ & 3&1000&1 & 3.9068	& -1.21e9\\
 \hline
SCCRM  &3 &1000& 1& 7.0298 & -3.59e8\\
 \hline
 CRM &3 &1000& 2& 17.7935  &	-1.57e7\\
 \hline
\hline
3PM& 10&100&5 & 0.4371 & -10088.9 \\
 \hline
 3PM $\parallel$ &10 &100&5 & 0.3246  & -10088.9\\
 \hline
  {\textbf{\color{red}A3PM}} & 10&100& 9&  {\textbf{\color{red}0.0106}}	&-4062.4\\
 \hline
 A3PM $\parallel$ &10 &100&9 & 0.2111&	-4062.4\\
 \hline
Cyclic  & 10&100& 1& 0.0683&	-251.9\\
 \hline
Cimmino  &10 &100&39 & 1.2540&	-6397.1\\
 \hline
Cimmino $\parallel$ &10 &100& 39& 1.2472&	-6397.1\\
 \hline
SCCRM  & 10& 100&8 &0.0884&	-1870.3\\
 \hline
 CRM & 10&100&  4&0.6351&	-1981.7\\
 \hline
 \end{tabular}
\caption{Comparison on several instances (for several values of $m$ and $n$) of the number of iterations, CPU time, and constraints violation at termination for 3PM, parallel 3PM (3PM $\parallel$), A3PM, parallel A3PM (A3PM $\parallel$), cyclic projections, Cimmino's method, parallel Cimmino's method (Cimmino $\parallel$), Successive Centralized CRM (SCCRM), and CRM using $\varepsilon=10^{-8}$.}\label{tableres}
\end{table}

\begin{table}
\centering
\begin{tabular}{|c|c|c|c|c|c|}
 \hline
 Method & m &  n & Iterations &  CPU time (s) & Constraints violation\\
 \hline
 3PM&10 &500&   3  &  6.7066 &	-1.3e6\\
 \hline
 3PM $\parallel$ & 10& 500& 3&    5.0711	&-1.3e6 \\
 \hline
 {\textbf{\color{red}A3PM}} & 10&500 &8&   	{\textbf{\color{red}0.3801}}	&-1.1e6   \\
 \hline
 A3PM $\parallel$ & 10&500&  8   &	0.4590	&-1.1e6 \\
 \hline
Cyclic  & 10&500& 1& 	1.9664	&-20651.6 \\
 \hline
Cimmino  & 10&500& 11&  11.8282&	-258241.8 	\\
 \hline
Cimmino $\parallel$ & 10&500& 11& 	8.2947&	-258241.8 \\
 \hline
SCCRM  & 10&500& 1&  1.1339	&-1.94e7 \\
 \hline
 CRM & 10&500& 4&  14.5173	&-1.51e6 \\
 \hline
 \hline 
   3PM& 10&1000 &2&   35.7312 &	-3.1e7	 \\
 \hline
 3PM $\parallel$ & 10&1000& 2 &  31.5069&	-3.1e7\\
 \hline
 {\textbf{\color{red}A3PM}} & 10&1000&  8	& {\textbf{\color{red}1.5014}}&	-6.8e6\\
 \hline
 A3PM $\parallel$ & 10&1000&8 & 1.6410&	-6.8e6 	\\
 \hline
Cyclic  & 10&1000& 1 & 	12.2368&	-2.2e7\\
 \hline
Cimmino  & 10&1000& 34& 179.3706&	-2.8e6\\
 \hline
Cimmino $\parallel$ & 10&1000 &34   	&161.6030&	-2.8e6\\
 \hline
SCCRM  & 10&1000& 	1& 6.5463&	-2.1e8\\
 \hline
 CRM & 10&1000& 3 & 	67.8992&	-6.9e7\\
 \hline
 \hline
3PM& 50 &500 & 3&  37.6726 &	-3.7e6 	  \\
 \hline
 3PM $\parallel$ &50 &500&3 & 30.4509&	-3.7e6	 \\
 \hline
 A3PM & 50&500&  9	&  3.3924&	-431986.7 \\
 \hline
 {\textbf{\color{red}A3PM $\parallel$}}  &50 &500&9 & {\textbf{\color{red}2.3589}} &	-431986.7 \\
 \hline
Cyclic  & 50&500& 1& 4.1694&	-121892.8\\
 \hline
Cimmino  & 50&500& 1& 19.7370  &	-7.8e7\\
 \hline
Cimmino $\parallel$ &50 &500& 1&  17.5321 &	-7.8e7\\
 \hline
SCCRM  & 50&500& 49 &9.6754 &	-73627.1 \\
 \hline
 CRM &50 &500&   5	& 102.5789 &	-612568.5\\
 \hline
 \hline
  3PM& 50&1000&2 & 156.5059&	-3.8e7 \\
 \hline
 3PM $\parallel$ &50 &1000& 2& 147.1472&	-3.8e7\\
 \hline
 A3PM & 50&1000&8 & 8.2272	&-4.2e6\\
 \hline
 A3PM $\parallel$ &50 &1000&8 & 7.4691 &	-4.2e6\\
 \hline
Cyclic  & 50&1000& 1& 17.6827 &	-2.0e7\\
 \hline
Cimmino  & 50&1000&15  & 605.2691 &	5.5e7\\
 \hline
Cimmino $\parallel$ & 50&1000& 29  &600.4700 &	2.9e7\\
 \hline
{\textbf{\color{red}SCCRM}}  &50 &1000& 1& {\textbf{\color{red}7.3082}}&	-2.1e8\\
 \hline
 CRM &50 &1000& 3& 412.9610	&-8.7e6\\
 \hline
\hline
3PM& 100&1000&2 & 358.8288&	-3.6e7 \\
 \hline
 3PM $\parallel$ &100 &1000&2 & 294.6009&	-3.6e7\\
 \hline
 A3PM & 100&1000& 8& 21.6928&	-593678.9\\
 \hline
 A3PM $\parallel$ &100 &1000&8 & 17.3567&	-593678.9\\
 \hline
Cyclic  & 100&1000& 1& 24.4503	&-4.4e7\\
 \hline
Cimmino  &100 &1000& 100& 3118.5956 &	3.4e7\\
 \hline
Cimmino $\parallel$ &100 &1000& 100& 2670.5723 &	3.4e7\\
 \hline
{\textbf{\color{red}SCCRM}}  & 100& 1000& 1&{\textbf{\color{red}7.1063}} &	-5.4e7\\
 \hline
 CRM & 100&1000&3 & 873.6313 &	-9.4e6\\
 \hline
 \end{tabular}
\caption{Comparison on several instances (for several values of $m$ and $n$) of the number of iterations, CPU time, and constraints violation at termination for 3PM, parallel 3PM (3PM $\parallel$), A3PM, parallel A3PM (A3PM $\parallel$), cyclic projections, Cimmino's method, parallel Cimmino's method (Cimmino $\parallel$), Successive Centralized CRM (SCCRM), and CRM using $\varepsilon=10^{-8}$.}\label{tableres1}
\end{table}

\begin{figure}
    \centering
    \begin{tabular}{cc}
\includegraphics[scale=0.4]{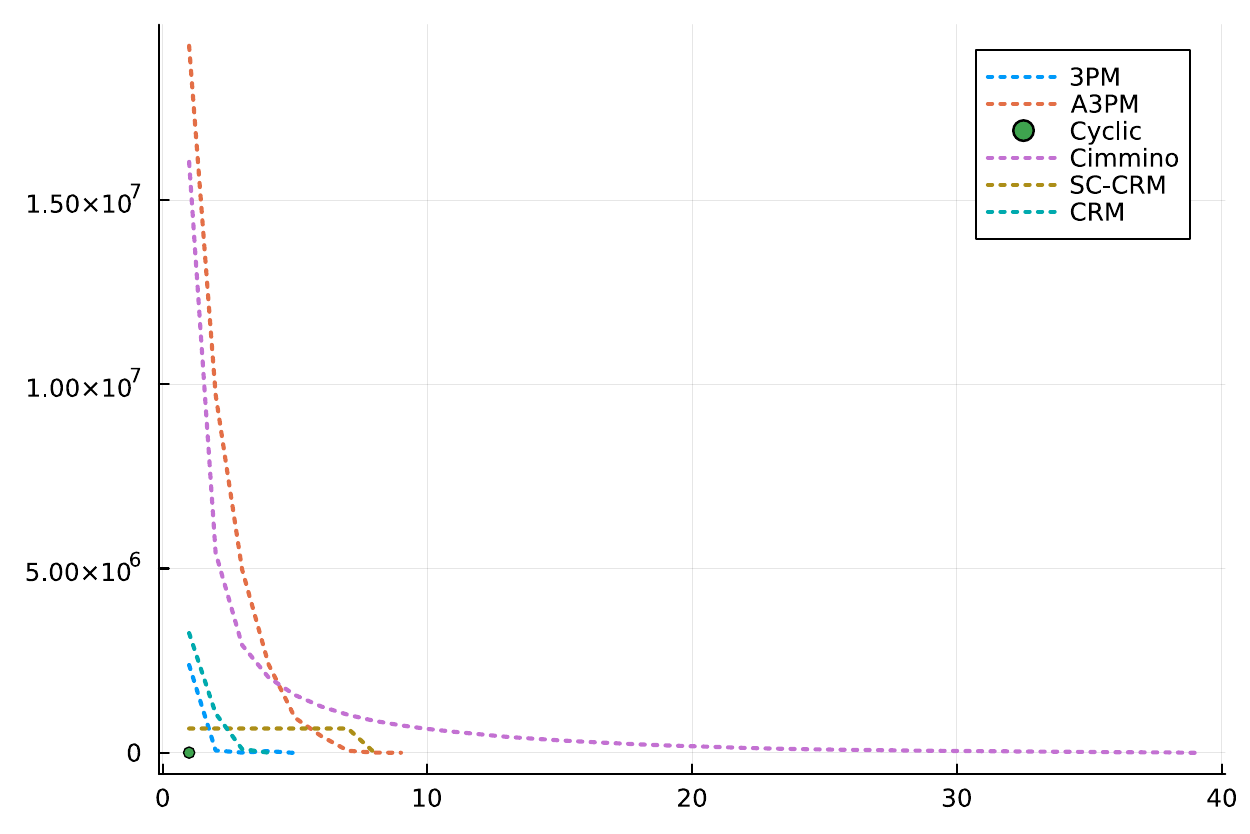}
¨& \includegraphics[scale=0.4]{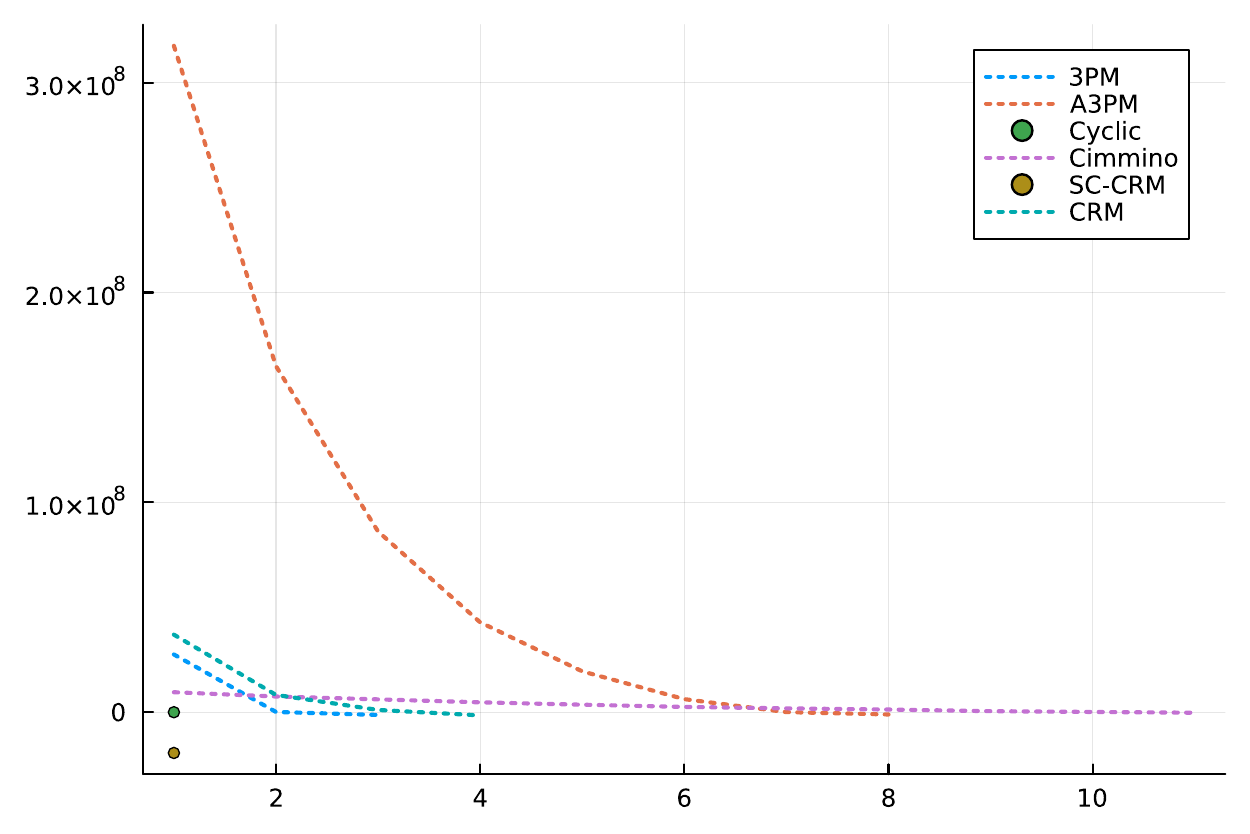}\\
$m=10, n=100$  & $m=10, n=500$\\
\includegraphics[scale=0.4]{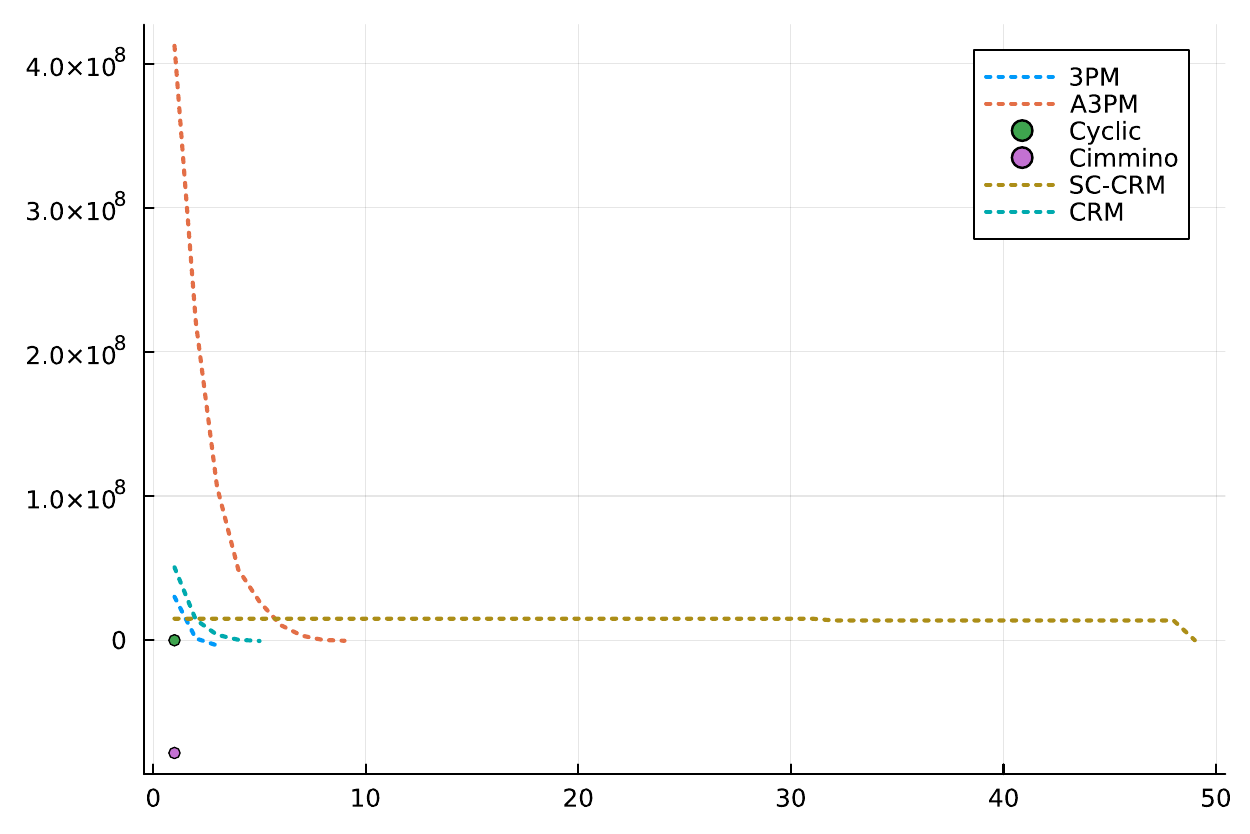}& 
\includegraphics[scale=0.4]{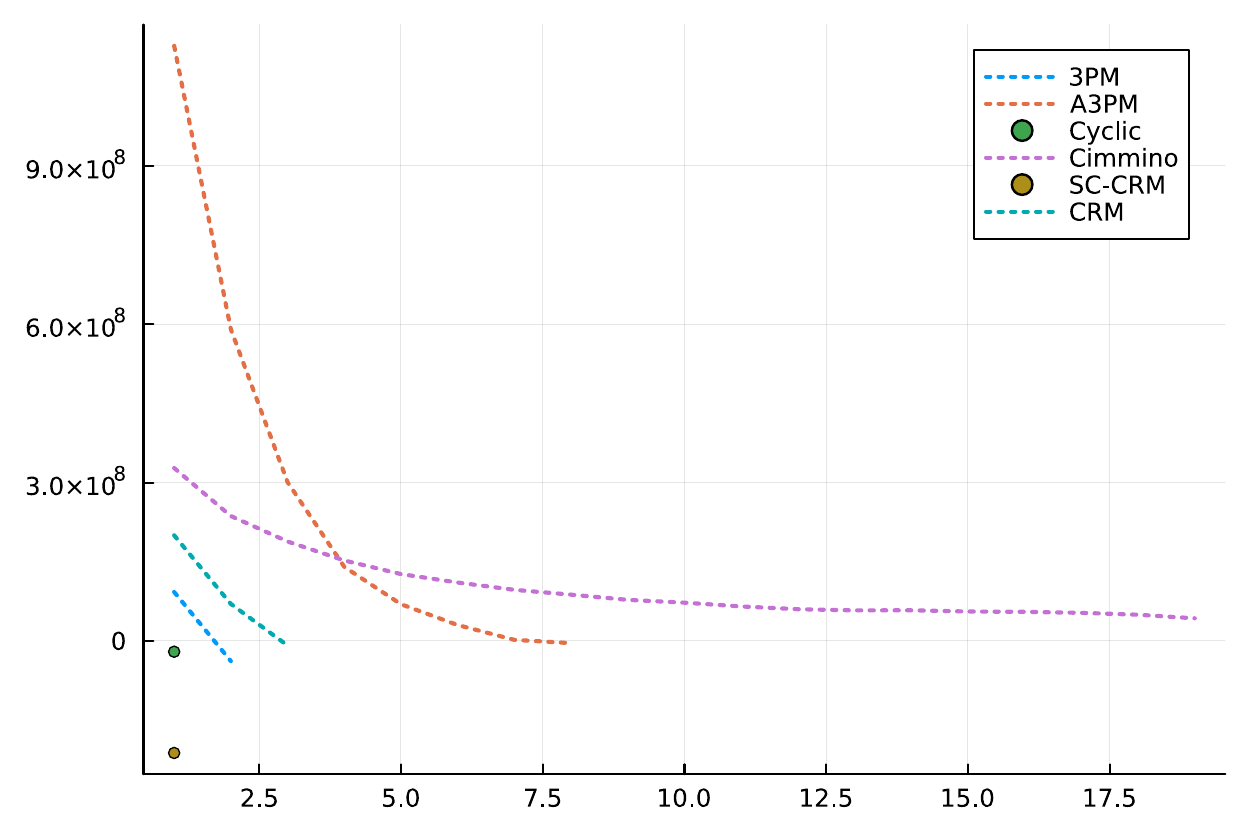}\\
$m=50, n=500$& $m=50, n=1000$
\end{tabular}
    \begin{tabular}{c}
\includegraphics[scale=0.4]{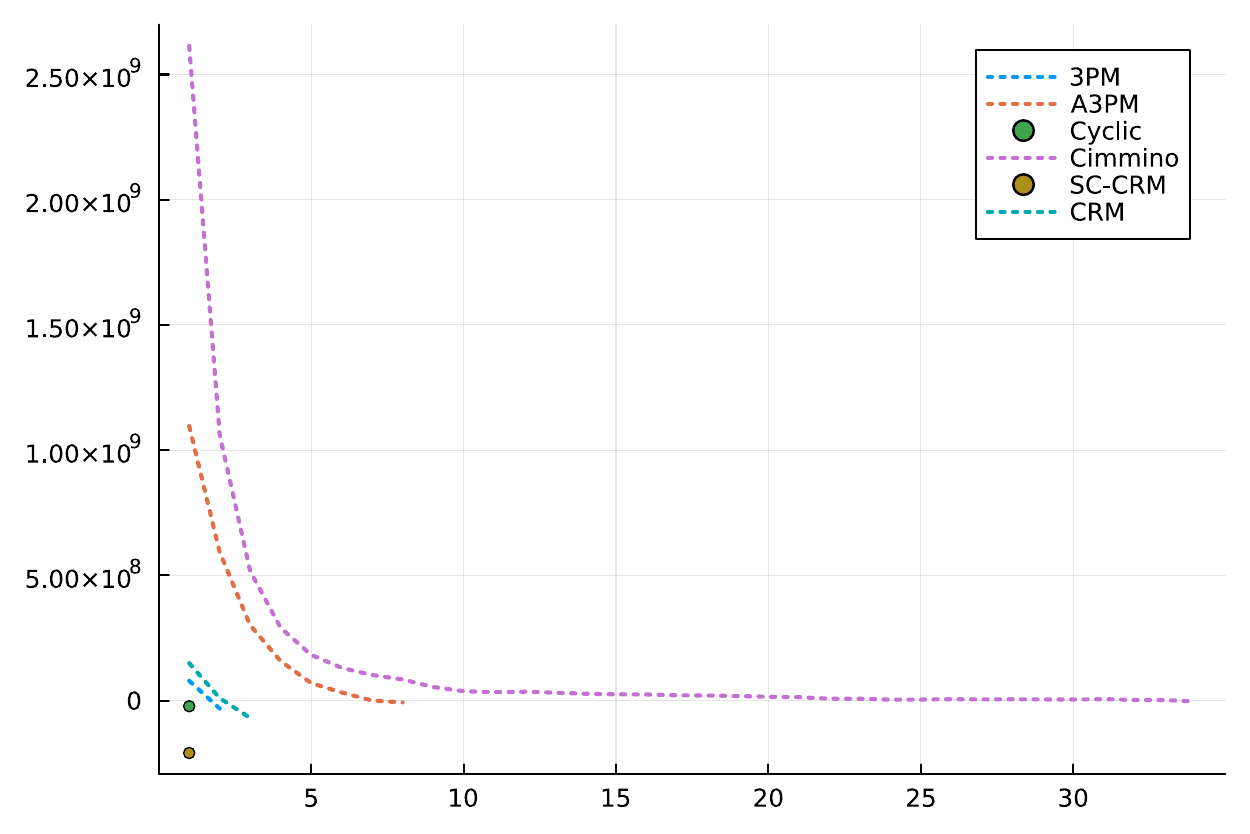}\\
$m=10, n=1000$
\end{tabular}
    \caption{Evolution of the constraints violations along iterations of the methods for several instances of the convex feasibility problem with several values of the number $m$ of sets and problem dimension $n$.}
    \label{fig:QCC2}
\end{figure}

We then consider among the 10 instances
chosen before the 7 largest 
and run each of the 9 methods ten times on these instances.
For each instance and each method, we report
in Table \ref{tableres2}.
the average CPU time (over the 10 runs)
required to find an
$\varepsilon$-approximate solution.
We report in red (resp. blue) and in bold the best (resp. second best)
method in terms of CPU time.
We see that A3PM and A3PM $\parallel$
are always either the best or second best
in terms of CPU time with A3PM being the
best in 6 of the 7 chosen problem instances.  For $m=50$ and $n=1000$, in the corresponding experiment from Table \ref{tableres1}, SCCRM was the
fastest while it is now (taking an average CPU time over 10 runs instead of using a single run of the method) the third best.
For the largest instance, SCCPM is still the best, needing 7.12 seconds
for convergence while the second best A3PM $\parallel$
needs 11.23 seconds. However, this experiment was still done with
the high accuracy $\varepsilon=10^{-8}$.
Since, every iteration of A3PM and A3PM $\parallel$
is very fast, we expect the CPU time with these methods
to decrease when $\varepsilon$ increases.
Therefore, we conducted the final following experiment: for the largest
instance where SCCRM was the fastest with $\varepsilon=10^{-8}$, we
consider three other values of $\varepsilon$: $\varepsilon=10^5$,
$\varepsilon=10^4$, and
$\varepsilon=10^{-3}$.
The CPU time with A3PM, A3PM $\parallel$, and SCCRM for this largest
instance and these values of $\varepsilon$ is reported in Table
\ref{tableres3}, with the quickest CPU time reported in bold red for each value of $\varepsilon$. We observe that for the largest value of
$\varepsilon$, both A3PM and A3PM $\parallel$ are now
quicker than SCCRM. Therefore, A3PM or A3PM $\parallel$
were the quickest as long as the problem instance is not very
large and for our largest instance if the accuracy is loose.

\begin{table}
\centering
\begin{tabular}{|c|c|c|c|c|c|c|c|c|c|}
 \hline
Problem sizes & 3PM & 3PM $\parallel$ & A3PM  & A3PM $\parallel$ &Cyclic&Cimmino&Cimmino $\parallel$&SCCRM&CRM\\ 
 \hline
$m=3, n=1000$ & 11.66 &8.74 & {\textbf{\color{blue}0.14}}& {\textbf{\color{red}0.13}}&5.08 &5.25 &3.49 &6.67 &17.02 \\   
 \hline
$m=10, n=100$ & 0.42 & 0.16&{\textbf{\color{blue}0.0068}} &{\textbf{\color{red}0.0051}} &0.065 &1.22 &1.23 &0.09 &0.62 \\   
 \hline
 $m=10, n=500$ & 6.16 & 4.58& {\textbf{\color{blue}0.33}}&  {\textbf{\color{red}0.20}} & 1.76  & 10.60   &7.39   & 1.19   &14.47 \\   
 \hline 
 $m=10, n=1000$& 34.87 &28.80  &{\textbf{\color{blue}1.15}} &{\textbf{\color{red}0.93}} &12.08 &177.74 &150.59 &6.64 & 67.69 \\   
 \hline
  $m=50, n=500$ & 28.27 &22.68 &{\textbf{\color{blue}2.04}} &{\textbf{\color{red}1.75}} &3.20 &14.83 &11,99 &6.90 &88.14\\   
 \hline
 $m=50, n=1000$ & 149.78 & 132.04& {\textbf{\color{blue}6.78}}&{\textbf{\color{red}6.65}} &16.24 &603.64 &601.01 &7.13 &364.77 \\   
 \hline
 $m=100, n=1000$ & 270.92 &242.79 &13.27 &  {\textbf{\color{blue}11.23}} &19.94 & 706.82& 646.04& {\textbf{\color{red}7.12}}&757.34 \\   
 \hline
 \end{tabular}
\caption{Comparison for several values of $m$ and $n$ of the 
average CPU time of all methods over 10 instances for $\varepsilon=10^{-8}$.}\label{tableres2}
\end{table}

\begin{table}
\centering
\begin{tabular}{|c|c|c|c|c|c|c|}
 \hline
Method & m & n &  $\varepsilon=100\,000$ & $\varepsilon=10\,000$   & $\varepsilon=10^{-3}$  & $\varepsilon=10^{-8}$\\ 
 \hline
SCCRM  &100&1000  &6.8563&{\color{red}7.0134}&{\color{red}7.2412}&{\color{red}7.1063}\\   
 \hline
A3PM &100&1000& 5.3124 &   9.8971  &  12.9859 &    21.6928\\
 \hline
 A3PM  $\parallel$&100&1000 &  {\color{red}3.4657} &  7.5412 & 11.1016&17.3567\\
 \hline 
 \end{tabular}
\caption{Comparison for $m=100$ and $n=1000$ of the 
CPU time for A3PM, A3PM $\parallel$,  and SCCRM for several values of $\varepsilon$.}\label{tableres3}
\end{table}

\section{Concluding Remarks}

We have proposed and analyzed the Parallel Polyhedral Projection Method (3PM) for solving convex feasibility problems. The method combines the computational advantages of parallel projections with a geometric correction step via projection onto a polyhedron, leading to a simple yet powerful two-phase algorithm.

Theoretical results include global convergence under no regularity assumptions, linear convergence under a standard error bound condition, and superlinear convergence when the sets and iterations of the algorithm exhibit additional geometric regularity. These results remain valid even when projections are computed approximately, making the method robust in practical scenarios.

Numerical experiments demonstrate that A3PM outperforms several classical projection schemes, especially in multi-set settings, and scales favorably as the number of constraint sets increases.

\bibliographystyle{plain}
\bibliography{references}

\end{document}